\documentclass[11pt]{article}

\usepackage[dvips]{graphicx}
\usepackage{amssymb}
\usepackage{amsmath}
\usepackage{amsthm}
\usepackage{latexsym}
\usepackage[english]{babel}
\usepackage{appendix}
\usepackage{mathtools}

\usepackage[usenames,dvipsnames]{color}

\newtheorem{theorem}{Theorem}[section]

\newtheorem{lemma}{Lemma}[section]
\newtheorem*{lemmaMaal}{Lemma (commutator estimates)}

\newtheorem{remark}{Remark}[section]

\newcommand{\HN}{{\mathbb{H}^N}}
\newcommand{\Sph}{S^{2N+1}}
\def \C {\mathbb{C}}
\def \R {\mathbb{R}}
\def \H {\mathbb{H}}

\def \LSph {\mathcal{A}_{2k}}
\def \LH {\mathcal{L}_{2k}}

\newcommand{\LpNorm}[2]{\left\| #1 \right\|_{L^{#2}}}
\newcommand{\SobNorm}[2]{\left\| #1 \right\|_{H^{#2}}}

\newcommand{\LpNormB}[3]{\left\| #1 \right\|_{L^{#2}\left(#3\right)}}
\newcommand{\SobNormB}[3]{\left\| #1 \right\|_{H^{#2}\left(#3\right)}}

\def \Cay {\mathcal{C}}
\def \Jacr {\Lambda_{\rho_n}}
\def \Jacs {\Lambda_{\sigma_n}}
\def \Jacc {\Lambda_{\Cay}}

\def \mlocSobH {H^{-k}_{\textnormal{loc}}\left(\HN\right)}
\def \locSobH {H^{k}_{\textnormal{loc}}\left(\HN\right)}
\def \SobSph {H^k\left(\Sph\right)}
\def \mSobSph {H^{-k}\left(\Sph\right)}

\def \dSph {\;\textnormal{d}v_{S}}
\def \dH {\;\textnormal{d}v_{H}}
\def \set {\mathcal{B}}
\def \e {\epsilon}

\def \N {\mathbb{N}}
\def \U {\mathbb{U}(N+1)}
\def \la {\langle}
\def \ra {\rangle}

\newcommand{\ol}[1]{\overline{#1}}

\begin{document}

\title{Palais-Smale sequences for the fractional CR Yamabe functional and multiplicity results}

\author{Chiara Guidi$^{(1)}$ \& Ali Maalaoui$^{(2)}$ \& Vittorio Martino$^{(3)}$}
\addtocounter{footnote}{1}
\footnotetext{Dipartimento di Matematica, Universit\`a di Bologna, piazza di Porta S.Donato 5, 40126 Bologna, Italy. E-mail address:
{\tt{chiara.guidi12@unibo.it}}}
\addtocounter{footnote}{1}
\footnotetext{Department of mathematics and natural sciences, American University of Ras Al Khaimah, PO Box 10021, Ras Al Khaimah, UAE. E-mail address:
{\tt{ali.maalaoui@aurak.ae}}}
\addtocounter{footnote}{1}
\footnotetext{Dipartimento di Matematica, Universit\`a di Bologna, piazza di Porta S.Donato 5, 40126 Bologna, Italy. E-mail address:
{\tt{vittorio.martino3@unibo.it}}}

\date{}
\maketitle

\vspace{5mm}

{\noindent\bf Abstract} {\small In this paper we consider the functional whose critical points are solutions of the fractional CR Yamabe type equation on the sphere. We firstly study the behaviour of the Palais-Smale sequences characterizing the bubbling phenomena and therefore we prove a multiplicity type result by showing the existence of infinitely many solutions to the related equation.}

\vspace{5mm}


\vspace{4mm}


\vspace{4mm}


\section{Introduction  and statement of the results}

\noindent
Let $N\geq 1$ and let $\Sph$ denote the (2N+1)-dimensional sphere, equipped with its standard CR structure. In this paper we consider the following energy functional
\begin{equation}\label{eq: functionalE}
E(u)=\frac{1}{2}\int_{\Sph}u\LSph u\dSph-\frac{1}{p^*}\int_{\Sph}|u|^{p^*}\dSph,\quad u\in H^k\left(\Sph\right)
\end{equation}
whose critical points satisfy the fractional CR Yamabe type equation
\begin{equation}\label{eq: problem on Sph}
\LSph u=|u|^{p^*-2}u \quad \text{on }\Sph,\quad u\in \SobSph.
\end{equation}
Here $k\in \R$ is a parameter such that $0<2k<Q:=2N+2$, $\LSph$ is the sub-elliptic intertwining operator of order $2k$ and $H^k(\Sph)$ is the related fractional Sobolev space, as defined for instance in \cite{FS,FGMT} (we will give all the rigorous definitions in Section 2); also, the exponent $p^*$ is the critical one for the embedding $H^k(\Sph)\hookrightarrow L^{p^*}(\Sph)$.\\
Just to fix the ideas, for instance when $k=1$, the operator $\mathcal{A}_2$ is nothing but the standard conformal sub-Laplacian on the sphere. Let us also notice that a similar functional $E_{\H}$ can be defined equivalently on the Heisenberg group $\HN$, via the Cayley transform $\mathcal{C}$, and the related equation on $\HN$ is given by
\begin{equation}\label{eq: problem on H}
\LH U=|U|^{p^*-2}U \quad \text{ on }\HN,\quad U\in D^k(\HN).
\end{equation}
We refer the reader to the next section for the definition of $D^k(\HN)$ and the relation between $\LSph$ and $\LH$.\\
These kind of conformally invariant operators were introduced in \cite{GG} and they can be seen as the CR counterpart to the GJMS operators defined in the Riemannian setting in \cite{GJMS}. Indeed, as in the Euclidean case, the existence of an infinite  family of explicit positive solutions (bubbles) to the previous equations is known, moreover due to the lack of compactness of the Sobolev embedding (which can be seen geometrically as the action of the conformal group), the functional $E$ does not satisfy the Palais-Smale condition.\\
However, a sharp Sobolev inequality has been proved by Frank and Lieb in \cite{FL}, showing that the extremals are exactly the bubbles. All these facts suggest that a characterization of the Palais-Smale sequences should be possible, making the bubbling phenomena completely explicit, as in the classic case.\\
Indeed, this is what we will prove in our first result. The proof is quite involved and delicate if compared to the standard case: this is due basically to the non-Euclidean setting, the degeneracy of the given operators and also the fractional nature of the problem, making it non-local. As it is commonly known in the standard setting, the bubbling phenomena occurs at a local scale which makes it harder to deal with in a non-local setting. In fact, even if a natural behavior is expected, any variational problem needs a careful analysis depending on the ambient manifold and the structure of the operators involved (see for instance \cite{I,MMde}).
In our particular case, in addition to the results in \cite{FL}, we will make use of some point-wise commutator estimates, which has been recently written specifically for these type of operators (see \cite{M}): as is the case of local operators, these kind of estimates are useful in order to study regularity properties, after localizing with cut-off functions (see for instance \cite{C, S, LS}). Therefore, we will prove the following
\begin{theorem}\label{thm: classification of PS sequences}
Let $u_n$ be a Palais-Smale sequence for the functional $E$ at level $c$. Then there exist $u_{\infty}$ a solution of \eqref{eq: problem on Sph}, $m$ sequences of points $\zeta_{n}^1,\dots,\zeta_{n}^m\in\Sph$ such that $\lim_{n\to\infty} \zeta_n^l=\zeta^l\in\Sph$ for $l=1,\dots, m$ and $m$ sequences of real numbers $R_{n}^1,\dots, R_{n}^m$ converging to zero, such that:
\begin{itemize}
\item[i)]$u_n=u_\infty +\sum_{l=1}^m v_n^l+o(1)$ in $\SobSph$
\item[ii)] $E(u_n)=E(u_{\infty})+\sum_{l=1}^m E_{\HN}(U_{\infty}^l)+o(1)$
\end{itemize}
where
\begin{equation*}
v_n^l=(\Lambda_{\sigma_n^l})^{\frac{Q-2k}{2Q}}\beta^lU_{\infty}^l\circ\sigma_{n}^l
\end{equation*}
with $\sigma_{n}^l=(\rho_{n}^l)^{-1}$, $\rho_{n}^l=\Cay\circ\delta_{R_n^l}\circ\tau_{\xi_n^l}$ and $\Lambda_{\sigma_n^l}$ denote half the absolute value of the Jacobian determinant of $\sigma_n^l$; $\beta^l$ are smooth compactly supported functions, such that $\beta^l\equiv 1$ on $B_{\frac{1}{4}}(\zeta^l)$, $\textnormal{supp}(\beta^l)\subset B_1(\zeta^l)$ and $U_{\infty}^l$ are solutions of \eqref{eq: problem on H}.
\end{theorem}

\noindent
Here the $B$'s stand for the balls on the sphere, $\delta$ and $\tau$ denote dilations and translations on the Heisenberg group, respectively. The proof of the previous theorem will be carried out in Section 3.\\
We want explicitly to mention that we recently found a paper on arXiv (\cite{LW}), in which the authors prove an existence result for the fractional Q-curvature problem on the three dimensional CR sphere: in their Lemma 2.1, they claim a behavior for Palais-Smale sequences along some flow lines, similar to our Theorem 1.1; the proof is missing, the authors cite a couple of papers, which in turn consider only local operators. To the best of our knowledge, we did not find any references dealing with these peculiar issues that we are considering in the present paper.

\noindent
Once we have characterized the Palais-Smale sequences, in Section 4 as main application we will prove a multiplicity result for equation (\ref{eq: problem on Sph}). We will argue by contradiction as in \cite{MMT2015}; in particular, with the help of some special groups of isometries, we will restrict the functional $E$ to some special subspaces and we will assume that the Palais-Smale condition fails: the action of the groups and the boundedness of the energy will lead to a contradiction. Therefore, a standard application of the minimax argument will give us the following result
\begin{theorem}\label{thm: infinitelysolutions}
There exist infinitely many solutions of \eqref{eq: problem on H} (or equivalently of \eqref{eq: problem on Sph}), distinct from the standard bubbles.
\end{theorem}

\noindent
Moreover, depending on the choice of the group of isometries, the existence of sign changing solutions can be shown. In this setting, we recall the paper \cite{MM}, where the existence of infinitely many sign changing solutions was proven for the standard CR-Yamabe equation on the sphere (here $k=1$), by following the idea of Ding \cite{D} combined with the action of the group of isometries generated by the Reeb vector field of the standard sphere. Moreover, recently in \cite{Kr}, under a technical assumption on the range of the parameter $k$, the author proved the existence of a number of sequences of sign-changing solutions of equation \eqref{eq: problem on Sph}, whose elements have mutually different nodal properties. The proof is based again on Ding's approach and on a iterative argument as in \cite{aubin}, starting from the result in \cite{MM} (the assumption on $k$ makes the iteration works fine). Also, in his Remark 3.2, the author wonders if his technical assumption could be removed in order to gain the compactness of some Sobolev embeddings: it seems that we can remove this assumption and still obtain existence of solutions.

\bigskip

\noindent
{\bf Acknowledgement}
The second author aknowledge the financial support of the Seed Grant of AURAK, No.: AAS/001/18, \emph{Critical Problems in the Sub-Elliptic Setting.}\\


\section{Definitions and notation}

\noindent
We identify the Heisenberg  group $\HN$ with $\C^N\times \R\simeq\R^{2N+1}$ with elements $w=(z,t)=(x+iy,t)\simeq(x,y,t)\in \R^N\times\R^N\times\R$ and group law
\begin{equation*}
w\cdot w'=(z,t)\cdot(z',t')=(z+z',t+t'+2\text{Im}(z\ol{z'}))\quad\forall\;  w,w'\in\HN,
\end{equation*}
where $\text{Im}$ denotes the imaginary part of a complex number and $z\ol{z'}$ is the standard Hermitian inner product in $\C^N$. Left translations on $\HN$ are defined by
\begin{equation*}
\tau:\HN\to\HN \qquad \tau_{w}(w')=w\cdot w'\quad \forall\;w\in\HN
\end{equation*}
and dilations are
\begin{equation*}
\delta_{\lambda}:\HN\to\HN \qquad \delta_{\lambda}(z,t)=(\lambda z,\lambda^2 t)\quad \forall\;\lambda>0.
\end{equation*}
The homogeneous dimension of $\HN$ with respect to $\delta_{\lambda}$ will be denoted by $Q=2N+2$. The natural distance that we will adopt in our setting is the Kor\'anyi distance, given by
\begin{equation*}
d((z,t),(z',t'))=\left(|z-z'|^4+(t-t'-2\textnormal{Im}(z\ol{z'}))^2\right)^{\frac{1}{4}}
\end{equation*}
and we denote by $B_R^w$ the ball of center $w$ and radius $R>0$ defined by the distance $d$. Moreover we denote by
\begin{equation*}
\theta_{\H}=\text{d}t+2\sum_{j=1}^{N}(x_i\text{d}y_i-y_i\text{d}x_i)
\end{equation*}
the standard contact form on $\HN$ and by $\dH$ the volume form associated to $\theta_{\H}$. The canonical basis of left invariant vector fields on $\HN$ is given by
\begin{equation*}
X_j=\frac{\partial}{\partial x_j}+2y_j\frac{\partial}{\partial t},\quad Y_j=\frac{\partial}{\partial y_j}-2x_j\frac{\partial}{\partial t},\quad T=\frac{\partial}{\partial t}, \quad j=1,\dots,N.
\end{equation*}
and the sub-Laplacian operator associated to this Carnot structure is given by
\begin{equation*}
\Delta_b=\frac{1}{4}\sum_{j=1}^{N}\left(X_j^2+Y_j^2\right).
\end{equation*}
The Heisenberg group can be identified with the unit sphere in $\C^{N+1}$ minus a point through the Cayley transform $\Cay:\HN\to S^{2N+1}\setminus \{(0,\dots,0,-1)\}$ defined as follows
\begin{equation*}
\Cay(z,t)=\left(\frac{2z}{1+|z|^2+it},\frac{1-|z|^2-it}{1+|z|^2+it}\right).
\end{equation*}
On the unit sphere $S^{2N+1}=\{\zeta\in\C^{N+1}:\; |\zeta|=1\}$ we consider the distance
\begin{equation*}
d(\zeta,\eta)^2=2|1-\zeta\ol{\eta}|,\quad \zeta,\eta\in \C^{N+1}
\end{equation*}
and we denote by $B_R(\zeta)\subset \Sph$ the ball of center $\zeta$ and radius $R>0$. With this definition of $d$, the relation between the distance of two points $w=(z,t)$, $w'=(z',t')$ in $\HN$ and the distance of their images $\Cay(w)$, $\Cay(w')$ in $\Sph$, is given by
\begin{equation*}
d(\Cay(w),\Cay(w'))=d(w,w')\left(\frac{4}{(1+|z|^2)^2+t^2}\right)^{\frac{1}{4}}\left(\frac{4}{(1+|z'|^2)^2+t'^2}\right)^{\frac{1}{4}}.
\end{equation*}
From this relation we deduce the following inclusions
\begin{equation}\label{eq: inclusions B and preimages}
\Cay^{-1}(B_R(\zeta))\supseteq B_{\frac{R}{2}}^{\Cay^{-1}(\zeta)}\quad\textnormal{for every } R>0
\end{equation}
and
\begin{equation*}
\Cay^{-1}(B_R(\mathcal{N}))\subseteq B_{R}^0\quad\textnormal{for every } 1\geq R>0
\end{equation*}
where $\mathcal{N}$ is the point $(1,0,\dots,0)\in\Sph$. On $\Sph$, we consider the standard contact form
\begin{equation*}
\theta_{S}=i\sum_{j=1}^{N+1}(\zeta_j\text{d}\ol{\zeta}_j-\ol{\zeta}_j\text{d}\zeta_j),
\end{equation*}
and we denote by $\dSph$ the volume form associated to $\theta_{S}$. The conformal sub-Laplacian is then
\begin{equation*}
\mathcal{A}_2=-\frac{1}{2}\sum_{j=1}^{N+1} (T_j\ol{T}_j+\ol{T}_jT_j) +\frac{N^2}{4}
\end{equation*}
where $T_j$ are the differential operators defined by
\begin{equation*}
T_j=\frac{\partial}{\partial \zeta_j}-\ol{\zeta_j}\sum_{k=1}^{N+1}\zeta_k\frac{\partial}{\partial\zeta_k},\; j=1,\dots, N+1.
\end{equation*}
Let $ \mathcal{H}_{j,l}$ be the space of harmonic polynomials on $\C^{N+1}$  homogeneous of degree $j$ and $l$ in variables $z$ and $\ol{z}$ respectively, restricted to $S^{2N+1}$. The Hilbert space $L^2(S^{2N+1})$ decomposes as $L^2(S^{2N+1})=\bigoplus_{j,l\geq 0} \mathcal{H}_{j,l}$ and we denote by $y_{j,l}^m$ an orthonormal basis  for the space $\mathcal{H}_{j,l}$, in particular we require $y_{j,l}^m$ to be eigenfunction for the conformal sub-Laplacian $\mathcal{A}_2$. Then, the conformal sub-Laplacian acts on $y_{j,l}^m$ as $\mathcal{A}_2 y_{j,l}^m=\lambda_j\lambda_l y_{j,l}^m $, where $\lambda_j=j+\frac{n}{2}$.
Let us fix $0<2k<Q$, and consider \begin{equation*}
u=\sum_{j,l}\sum_{m=1}^{\text{dim}(\mathcal{H}_{j,l})}c_{j,l}^m(u)y_{j,l}^m \in L^2(S^{2N+1}).
\end{equation*}  We define the operator
\begin{equation*}
\mathcal{A}^{k}u=\sum_{j,l}\sum_{m=1}^{\text{dim}(\mathcal{H}_{j,l})}(\lambda_j\lambda_l)^{\frac{k}{2}}c_{j,l}^m(u)y_{j,l}^m
\end{equation*}
where $\text{dim}(\mathcal{H}_{j,l})=\frac{(j+N-1)!(l+N-1)!(j+l+N)}{N!(N-1)!j!l!}$ is the dimension of $ \mathcal{H}_{j,l}$. Moreover, we define the Sobolev space
\begin{equation*}
H^k(S^{2N+1})=\left\{u\in L^2(\Sph):\mathcal{A}^{k}u\in L^2(\Sph)\right\}
\end{equation*}
with inner product
\begin{equation*}
\la u,v\ra_{k}=\int_{\Sph}\mathcal{A}^{k}u\ol{\mathcal{A}^{k}v}\;\dSph
\end{equation*}
and norm
\begin{equation*}
\|u\|_{k}=\la u,u\ra_{k}^{\frac{1}{2}}=\left(\sum_{j,l}\sum_{m=1}^{\text{dim}(\mathcal{H}_{j,l})}(\lambda_j\lambda_l)^k|c_{j,l}^m(u)|^2\right)^{\frac{1}{2}}.
\end{equation*}
We consider the intertwining operator $\LSph$ on $\Sph$ defined, up to multiplicative constants, by \begin{equation}\label{eq: intertwining}
\text{Jac}_{\tau}^{\frac{Q+2k}{2Q}}(\LSph u)\circ\tau=\LSph\left(\text{Jac}_{\tau}^{\frac{Q-2k}{2Q}}(u\circ\tau)\right)\quad \forall \tau\in\text{Aut}(\Sph),\; u\in C^{\infty}(\Sph).
\end{equation}
Moreover from now on we endow $\SobSph$ with the inner product
\begin{equation*}
\la u,v\ra_{H^k}=\sum_{j,l}\sum_{m=1}^{\text{dim}(\mathcal{H}_{j})}(\lambda_j(k))^{2k}c_{j}^m(u)\ol{c_{j}^m(v)}=\int_{\Sph}\ol{v}\LSph u\;\dSph
\end{equation*}
with \begin{equation*}
\lambda_j(k)=\frac{\Gamma\left(\frac{Q+2k}{4}+j\right)}{\Gamma\left(\frac{Q-2k}{4}+j\right)}, \quad j=0,1,\dots
\end{equation*}
and norm $\SobNorm{u}{k}=\left(\int_{\Sph}\ol{u}\LSph u\;\dSph\right)^{\frac{1}{2}} $ which is equivalent to $\|u\|_k$. The dual of $\SobSph$ will be denoted by $H^{-k}$. In $\HN$ the symbol of the intertwining operators is defined, up to a multiplicative constant, by
\begin{equation*}
\LH=|2T|^k\frac{\Gamma\left( \frac{-\Delta_{b}}{|2T|}+\frac{1+k}{2}\right)}{\Gamma\left( \frac{-\Delta_{b}}{|2T|}+\frac{1-k}{2}\right)};
\end{equation*}
we choose the multiplicative constant to be equal $1$ so that we recover $\mathcal{L}_2=-\Delta_{b}$ and $\mathcal{L}_4=(-\Delta_{b})^2-T^2.$  Hereafter we consider only real valued functions. The quadratic form associated to $\LH$ will be denoted by $a_{2k}$ :
\begin{equation*}
a_{2k}	\left(U\right)=\int_{\HN}U\LH U\dH
\end{equation*}
and we define the space
\begin{equation*}
D^k(\HN)=\left\{\ U\in L^{\frac{2Q}{Q-2k}}(\HN):\; a_{2k}<+\infty \right\}.
\end{equation*}
The operators $\LSph$ and $\LH$ are related by the following identity
\begin{equation}\label{eq: relationship LH LSph}
\LH\left(\Jacc^{\frac{Q-2k}{2Q}}(u\circ \Cay)\right)=\Jacc^{\frac{Q+2k}{2Q}}(\LSph u)\circ \Cay \quad \forall u\in \SobSph
\end{equation}
where $\Jacc$ is twice the absolute value of the Jacobian determinant of the Cayley transform
\begin{equation*}
\Jacc=\frac{2^Q}{\left((1+|z|^2)^2+t^2\right)^{N+1}}.
\end{equation*}
We recall now the following sharp Sobolev inequality that was proved by Frank and Lieb in \cite{FL}
\begin{equation}\label{eq: Sobolev inequality Sph}
\left(\int_{\Sph}|u|^{\frac{2Q}{Q-2k}}\dSph\right)^{\frac{Q-2k}{Q}}\leq C_S\int_{\Sph}u\LSph u \dSph
\end{equation}
where
\begin{equation}\label{eq: Sobolev constant Sph}
C_S(k,N)=\frac{\Gamma\left(\frac{N+1-k}{2}\right)^2}{\Gamma\left(\frac{N+1+k}{2}\right)^2}( \omega_{2N+1}2^{2N+1}N!)^{-\frac{2k}{Q}},
\end{equation}
$\omega_{2N+1}$ is the measure of $S^{2N+1}$ and
\begin{equation*}
p^*=\frac{2Q}{Q-2k}
\end{equation*}
is the critical exponent. Indeed the embedding
\begin{equation}\label{emb}
H^k(\Sph)\hookrightarrow L^{p^*}(\Sph)
\end{equation}
is continuous but not compact and this is due to the scale invariance of the norms, induced by the action of the conformal group. Also, we will denote by $\bar p=(p^{*})'=\frac{2Q}{Q+2k}$ and it follows from $(\ref{emb})$ that
$$L^{\bar{p}}(\Sph)\hookrightarrow H^{-k}(\Sph).$$
For $\Omega\subset\HN$ open and bounded we denote by $H^k_0(\Omega)$ the closure of $C_0^{\infty}(\Omega)$ with respect to the norm
\begin{equation*}
\|U\|_{H^k_0(\Omega)}=\left(\int_{\Omega}U\LH U\dH\right)^{\frac{1}{2}}
\end{equation*}
and it holds
\begin{align*}
H^k_0(\Omega)&\hookrightarrow L^{p^*}(\Omega).
\end{align*}
Optimizer functions for \eqref{eq: Sobolev inequality Sph} are images through the Cayley transform of functions of the type $\lambda^{\frac{2k-Q}{2}}\omega\circ\delta_{\lambda^{-1}}\circ\tau_{\xi^{-1}}$ where
\begin{equation}\label{eq: optimizer function}
\omega(z,t)=\frac{c(Q)}{\left((1+|z|^2)^2+t^2\right)^{\frac{Q-2k}{4}}} \;,
\end{equation}
for a suitable positive constant $c(Q)$ (see \cite{FL}). These functions satisfy the equation
\begin{equation*}
\LH U=|U|^{p^*-2}U \quad \text{on }\HN\quad U\in D^k(\HN),
\end{equation*}
hence they are critical points for the energy functional $E_{\H}$ defined on $D^{k}(\HN) $ by
\begin{equation*}
E_{\H}(U)=\frac{1}{2}\int_{\HN}U\LH U\dH-\frac{1}{p^*}\int_{\HN}|U|^{p^*}\dH.
\end{equation*}
In fact the functions $\omega_{\lambda,\xi}=\lambda^{\frac{2k-Q}{2}}\omega\circ\delta_{\lambda^{-1}}\circ\tau_{\xi^{-1}}$ are (the only) ground state solutions of $E_{\H}$.


\section{Classification of the Palais-Smale sequences}

\noindent
Let $H$ be an Hilbert space, a sequence $\{x_n\}_{n\in\N}\subseteq H$ is called a Palais-Smale (PS) sequence for $F\in C^1(H,\mathbb{R})$ at level $c$ if $F(x_n)\to c$ and $\nabla F(x_n)\to 0$.  $F$ is said to satisfy the Palais-Smale condition if any (PS) sequence admits a converging subsequence.\\

\noindent
Now, we begin the proof of our main result, that is Theorem \ref{thm: classification of PS sequences}.


\begin{lemma}\label{lem: boundness of PS sequences}
Every (PS) sequence $u_n$ for $E$ is bounded.
\end{lemma}
\begin{proof}
Let $u_n$ be a (PS) sequence for $E$ at level $c$ i.e.
\begin{equation*}
E(u_n)\to c,\quad dE(u_n)\to 0\; \text{in } \mSobSph.
\end{equation*}
Therefore we have
\begin{align*}
2c+o(1)+o(1)\|u_n\|_{H^k}&\geq 2E(u_n)-\la dE(u_n),u_n \ra\\
&=\left(\frac{p^*-2}{p^*}\right)\int_{\Sph} |u_{n}|^{p^*}\dSph,
\end{align*}
hence
\begin{align*}
\|u_n\|_{H^k}^2&=2E(u_n)+\frac{2}{p^*}\int_{\Sph}|u_{n}|^{p^*}\dSph\\
&\leq 2c+o(1)+\frac{2}{p^*-2}\left(2c+o(1)+o(1)\|u_n\|_{H^k}\right).
\end{align*}
It follows that $u_n$ is bounded in $\SobSph$.
\end{proof}

\noindent
The result above implies that, up to a subsequence, there exists a function $u_{\infty}\in\SobSph$ such that
\begin{align}\label{eq: weak conv u_n Hk Sph}
&u_n\rightharpoonup u_{\infty}\quad \text{weakly in }\SobSph\;,\\ \label{eq: strong conv u_n Lp Sph}
&u_n\to u_{\infty}\quad \text{strongly in } L^p(\Sph) \quad \text{for } 1\leq p<p^* \;.
\end{align}
Moreover, $u_{\infty}$ is a weak solution to \eqref{eq: problem on Sph}. Indeed, since $u_n$ is a (PS) sequence for $E$, for any $\varphi\in\SobSph$ we have
\begin{equation*}
\int_{\Sph} \varphi \LSph u_n\dSph=\int_{\Sph} \varphi |u_n|^{p^*-2}u_n\dSph+o(1)
\end{equation*}
as $n\to\infty$, and by \eqref{eq: strong conv u_n Lp Sph} and \eqref{eq: weak conv u_n Hk Sph} we have respectively
\begin{align*}
\int_{\Sph}\varphi \LSph u_n \dSph&\to \int_{\Sph}\varphi \LSph u_{\infty} \dSph\\
\int_{\Sph} \varphi |u_n|^{p^*-2}u_n\dSph&\to \int_{\Sph} \varphi |u_{\infty}|^{p^*-2}u_{\infty}\dSph
\end{align*}
showing that $u_{\infty}$ weakly satisfies \eqref{eq: problem on Sph}. We set $v_n=u_n-u_{\infty}$, with this notation we have the following
\begin{lemma}\label{lem: the energy splits}
The sequence $v_n$ is a (PS) sequence for $E$. More precisely, it holds
\begin{align*}
E(v_n)=E(u_n)-E(u_{\infty})+o(1)\\
\intertext{and}
dE(v_n)\to 0, \quad \text{in }\mSobSph.
\end{align*}
\end{lemma}
\begin{proof}We have
\begin{align*}
2E(u_n)&=\int_{\Sph} (v_n+u_{\infty})\LSph (v_n+u_{\infty})\dSph-\frac{2}{p^*}\int_{\Sph}|v_n+u_{\infty}|^{p^*}\dSph\\
&=2E(v_n)+2E(u_{\infty})+2\la dE(u_{\infty}),v_n \ra+2\int_{\Sph}|u_{\infty}|^{p^*-2}u_{\infty}v_n\dSph+\\
&\qquad+\frac{2}{p^*}\int_{\Sph}|v_n|^{p^*}+|u_{\infty}|^{p^*}-|v_n+u_{\infty}|^{p^*}\dSph.
\end{align*}
Since $v_n\to 0$ in $L^p$ for every $1\leq p<p^*$, we have  $\int_{\Sph}|u_{\infty}|^{p^*-2}u_{\infty}v_n\dSph=o(1)$ as $n\to \infty$; moreover $dE(u_{\infty})=0$ , so that it remains to show that the last integral in the expression above goes to $0$ as $n\to \infty$.
It is possible to choose a big enough positive constant $C$ such that
\begin{equation*}
|\Phi_n|:=\left||v_n+u_{\infty}|^{p^*}-|v_n|^{p^*}-|u_{\infty}|^{p^*}\right|\leq C|v_n|^{p^*-1}|u_{\infty}|+C|v_n||u_{\infty}|^{p^*-1}.
\end{equation*}
Hence, by the H\"older inequality
\begin{align*}
\int_{\Sph}|\Phi_n|\dSph&=\int_{\Sph\setminus M_{\e}}|\Phi_n|\dSph+\int_{ M_{\e}}|\Phi_n|\dSph\\
&\leq \int_{\Sph\setminus M_{\e}}|\Phi_n|\dSph +C\left(\int_{ M_{\e}}|v_n|^{p^*}\dSph\right)^{\frac{p^*-1}{p^*}}\left(\int_{M_{\e}}|u_{\infty}|^{p^*}\dSph\right)^{\frac{1}{p^*}}\\
&\quad+ C\left(\int_{ M_{\e}}|u_{\infty}|^{p^*}\dSph\right)^{\frac{p^*-1}{p^*}}\left(\int_{M_{\e}}|v_n|^{p^*}\dSph\right)^{\frac{1}{p^*}}.
\end{align*}
Here $M_{\e}\subset\Sph$, defined for any $\e>0$ by the Egorov theorem, is such that $|\Sph\setminus M_{\e}|<\e$ and $v_n$ converges to $0$ uniformly on $M_{\e}$. So that, the first integral in the expression above converges to $0$ as $n\to \infty$, while the other two terms go to $0$ as $\e\to 0$, uniformly in $n$. Therefore we get the desired energy estimate. Now we prove that for any $\varphi \in \SobSph$ with $\SobNormB{\varphi}{k}{\Sph} \leq1$, it holds
\begin{equation*}
\la dE(v_n),\varphi\ra=o(1)\quad \text{as } n\to\infty.
\end{equation*}
We have
\begin{align*}
\la dE(u_n),\varphi\ra&=\int_{\Sph} \varphi\LSph(v_n+u_{\infty})\dSph-\int_{\Sph} |u_n|^{p^*-2}u_n\varphi\dSph\\
&=\la dE(v_n),\varphi\ra+\la dE(u_{\infty}),\varphi\ra+\\
&\qquad -\int_{\Sph}\left(|u_n|^{p^*-2}u_n-|v_n|^{p^*-2}v_n-|u_{\infty}|^{p^*-2}u_{\infty}\right)\varphi\dSph.
\end{align*}
Since $\la dE(u_{\infty}),\varphi\ra=0$ and $\la dE(u_n),\varphi\ra=o(1)$ as $n\to\infty$, it remains to show that the last integral in the equality above converges to $0$ as $n\to \infty$.
Again, for a big enough positive constant $C$ we have
\begin{align*}
|\Psi_n|:=\left||u_n|^{p^*-2}u_n-|v_n|^{p^*-2}v_n-|u_{\infty}|^{p^*-2}u_{\infty}\right|\leq C|v_n|^{p^*-2}|u_{\infty}|+C|v_n||u_{\infty}|^{p^*-2}
\end{align*}
and by the H\"older inequality and Egorov theorem
\begin{align*}
\left|\int_{\Sph}\Psi_n|\varphi|\dSph\right|&\lesssim \int_{\Sph}|v_n|^{p^*-2}|u_{\infty}||\varphi|\dSph+\int_{\Sph}|v_n||u_{\infty}|^{p^*-2}\varphi\dSph\\
&\lesssim \LpNorm{|v_n|^{p^*-2}|u_{\infty}| }{\frac{p^*}{p^*-1}}\LpNorm{\varphi }{p^*}+\LpNorm{|v_n||u_{\infty}|^{p^*-2} }{\frac{p^*}{p^*-1}}\LpNorm{\varphi }{p^*}\\
&\lesssim  \LpNorm{|v_n|^{p^*-2}|u_{\infty}| }{\frac{p^*}{p^*-1}}+\LpNorm{|v_n||u_{\infty}|^{p^*-2} }{\frac{p^*}{p^*-1}}\\
&\lesssim \LpNormB{|v_n|^{p^*-2}|u_{\infty}| }{\frac{p^*}{p^*-1}}{\Sph\setminus M_{\e}}+\LpNormB{|v_n||u_{\infty}|^{p^*-2} }{\frac{p^*}{p^*-1}}{\Sph\setminus M_{\e}}\\
&\quad +\LpNormB{|v_n|^{p^*-2}|u_{\infty}| }{\frac{p^*}{p^*-1}}{M_{\e}}+\LpNormB{|v_n||u_{\infty}|^{p^*-2} }{\frac{p^*}{p^*-1}}{M_{\e}}\\
&=o(1)\quad \text{as } n\to \infty \text{ and }\e\to 0,\textnormal{ uniformly in }n.
\end{align*}
This concludes the proof.
\end{proof}

\begin{lemma}\label{lem: critical energy level}
Let $u_n$ be a (PS) sequence at level $c<\frac{k}{Q} C_{S}^{-\frac{Q}{2k}}$, then $u_n$ converges strongly to $0$ in $\SobSph$. Here $C_S$ is the Sobolev constant defined in \eqref{eq: Sobolev constant Sph}.
\end{lemma}
\begin{proof}
By the Sobolev inequality \eqref{eq: Sobolev inequality Sph}, we have $\LpNorm{u_n}{p^*}^{p^*}\leq C_S^{\frac{p^*}{2}}\SobNorm{u_n}{k}^{p^*}$, so that
\begin{align*}
o(1)&=\int_{\Sph}u_n\LSph u_n\dSph-\int_{\Sph}|u_n|^{p^*}\dSph\\
&\geq \SobNorm{u_n}{k}^2\left(1-C_S^{\frac{p^*}{2}}\SobNorm{u_n}{k}^{p^*-2}\right).
\end{align*}
Now, following the argument given in Lemma \ref{lem: boundness of PS sequences},  we notice that the choice of $c$ in the statement implies
\begin{align*}
\SobNorm{u_n}{k}^2&\leq \frac{p^*}{p^*-2}2c+o(1)\\
&< C_S^{-\frac{Q}{2k}}+o(1).
\end{align*}
This, for $n$ big enough, ensures the positivity of the factor $1-C_S^{\frac{p^*}{2}}\SobNorm{u_n}{k}^{p^*-2}$, concluding the proof.
\end{proof}

\noindent
Hereafter we assume the (PS) sequence $(u_n)_{n\in \N}$ converges weakly to $0$  in $\SobSph$ and strongly in $L^p(\Sph)$ for $1\leq p<p^*$. Moreover, since we want to investigate the behavior of $(u_n)$ when the (PS) condition is not satisfied, we will assume that $u_n$ does not converge strongly to $0$ in $\SobSph$. For any $\e_0>0$, we define
\begin{equation*}
\Sigma_{\e_0}=\left\{ \zeta\in\Sph:\; \liminf_{r\to 0}\liminf_{n\to\infty}\int_{B_r(\zeta)}|u_n|^{p^*}\dSph\geq \e_0\right\}.
\end{equation*}
In the sequel we will need to localize our equation, therefore we will use some commutator estimates. For a given $k \in (0,\frac{Q}{2})$, we define the $3$-commutator $H^{\LH}(\cdot,\cdot)$ by
$$H^{\LH}(u,v)=\LH(uv)-u\LH(v)-v\LH(u),$$
and we let $R_{s}$ denote the Riesz potential on $\mathbb{H}^{N}$ (see Appendix). Then the following Lemma \cite{M} holds
\begin{lemmaMaal}\label{lem: commutator}
Let $0<2k<Q$ and $\epsilon>0$. Given $\tau_{1}$ and $\tau_{2}$ in $(\max\{0,2k-1\},2k]$ such that $\tau_{1}+\tau_{2}>2k$, there exists $L\in \mathbb{N}$, $s_{j,1}\in (0,\tau_{1})$, $s_{j,2}\in (0,\tau_{2})$, for $j=1,\cdots, L$, satisfying $\tau_{1}+\tau_{2}-s_{j,1}-s_{j,2}-2k\in [0,\epsilon)$ such that
\begin{equation}
|H^{\LH}(u,v)|(x)\lesssim \sum_{j=1}^{L}R_{\tau_{1}+\tau_{2}-s_{j,1}-s_{j,2}-2k}\Big(R_{s_{j,1}}|\mathcal{L}_{\tau_{1}}u|R_{s_{j,2}}|\mathcal{L}_{\tau_{2}}v|\Big)(x).
\end{equation}
\end{lemmaMaal}

\noindent
Next, we state the first Lemma characterizing the concentration set:
\begin{lemma}\label{lem: concentration}
There exists $\e_0>0$ such that if $\zeta_0\notin\Sigma_{\e_0}$, then for a small enough $r>0$, we have $u_n\to 0$ in $H^k\left(B_r(\zeta_0)\right)$.
\end{lemma}
\begin{proof}
Suppose by contradiction that for every $\e>0$ there exists $\zeta_0\notin\Sigma_{\e}$ such that $u_n$ does not converges to $0$ in $H^k\left(B_r(\zeta_0)\right)$ for every $r>0$. Notice that $\zeta_0\notin\Sigma_{\e}$ implies the existence of a radius $r>0$ such that \begin{equation}\label{eq: int Br un <e}
\int_{B_{2r}(\zeta_0)}|u_n|^{p^*}\dSph<\e.
\end{equation}
Since $u_n$ is a (PS) sequence, there exists a sequence $\delta_n\in \mSobSph$ converging to $0$ in $\mSobSph$, such that
\begin{equation*}
\LSph(u_n)=|u_n|^{p^*-2}u_n+\delta_n.
\end{equation*}
Since we want to localize around $\zeta_{0},$ we consider the Cayley transform $\Cay$ where we set the north pole as $\zeta_0$. We notice that for any couple of functions $u,v$ it holds
$$
\LSph(uv)=\Jacc^{-\frac{1}{\bar p}}[\LH(\Jacc^{\frac{1}{p^{*}}}(uv)\circ \Cay)]\circ \Cay^{-1}
$$
and
$$\LH(\Jacc^{\frac{1}{p^{*}}}(uv)\circ \Cay)=v\circ \Cay\LH(\Jacc^{\frac{1}{p^{*}}}u\circ \Cay)+\Jacc^{\frac{1}{p^{*}}}u\circ \Cay\LH(v\circ \Cay)+H^{\LH}(\Jacc^{\frac{1}{p^{*}}}u\circ \Cay,v\circ \Cay).$$
We go back now to our Palais-Smale sequence and we compute
\begin{align*}
\LSph(\eta u_n)&=\eta\LSph( u_n)+\text{l.o.t.}\\
&=\eta|u_n|^{p^*-2}u_n+\eta\delta_n+\text{l.o.t}.
\end{align*}
Here $\eta$ is a  smooth cut off function  with $\textnormal{supp}(\eta)\subseteq B_{2r}(\zeta_0)$ and $\eta\equiv 1$ on $B_r(\zeta_0)$. So we first estimate $\|\text{l.o.t}\|_{H^{-k}}$. We have
\begin{align}\label{eq: splitting}
|H_{k}(\Jacc^{\frac{1}{p^{*}}}u_{n}\circ \Cay,\eta\circ \Cay)|&\lesssim \sum_{i=1}^{L}R_{2k-s_{i}-t_{i}}\Big(R_{t
_{i}}(\LH(\Jacc^{\frac{1}{p^{*}}}u_{n}\circ \Cay))R_{s_{i}}(\LH(\eta\circ \Cay))\Big)\notag\\
&\lesssim \sum_{i=1}^{L}R_{2k-s_{i}-t_{i}}\Big(R_{t_{i}}((\tilde{u}_{n}+\delta_{n}) \circ \Cay)R_{s_{i}}(\LH(\eta\circ \Cay))\Big)
\end{align}
where $\tilde{u}_{n}=|\Jacc^{\frac{1}{p^{*}}}u_{n}|^{p^{*}-2}\Jacc^{\frac{1}{p^{*}}}u_{n}$. Since the terms of the sumation above are similar in nature we will give here the proof for a single term. Since $\tilde{u}_{n}$ is bounded in $L^{\bar p}$, we have that $R_{t}(\tilde{u}_{n}\circ \Cay)$ converges strongly to zero (up to a subsequence) in $L^{q}_{loc}$ for $\frac{1}{q}>\frac{1}{\bar p}-\frac{t}{Q}$. So we fix $R>0$ big enough. Then we have
$$\left\|R_{2k-s-t}\Big(R_{t}(\tilde{u}_{n}\circ \Cay)\chi_{B_{R}^0}R_{s}(\LH(\eta\circ \Cay))\Big)\right\|_{L^{\bar p}}\lesssim\|R_{t}(\tilde{u}_{n}\circ \Cay)\|_{L^{q}(B_{R}^0)}\|\LH(\eta\circ \Cay)\|_{L^{p}},$$
where $\frac{1}{\bar p}=\frac{1}{q}+\frac{1}{p}-\frac{2k-t}{Q}$. Hence
$$\left\|R_{2k-s-t}\Big(R_{t}(\tilde{u}_{n}\circ \Cay)\chi_{B_{R}^0}R_{s}(\LH(\eta\circ \Cay))\Big)\right\|_{L^{\bar p}}\to 0.$$
Outside $B_{R}^0$, we have that
$$\chi_{\mathbb{H}^{N}\setminus B_{R}^0}R_{s}(\LH(\eta\circ \Cay))(x)\lesssim \frac{1}{|x|^{Q+2k-s}}.$$
Thus,
$$\left\|R_{2k-s-t}\Big(R_{t}(\tilde{u}_{n}\circ \Cay)\chi_{\mathbb{H}^{N}\setminus B_{R}^0}R_{s}(\LH(\eta\circ \Cay))\Big)\right\|_{L^{\bar p}}\lesssim\|\tilde{u}_{n}\circ \Cay\|_{L^{\bar p}}\frac{1}{R^{Q+2k-s-t}}$$
and since
$$\|\tilde{u}_{n}\circ \Cay\|_{L^{\bar p}}\lesssim \|u_{n}\|_{L^{p^{*}}}^{p^{*}-1}\leq C,$$
by letting first $n\to \infty$ then $R\to \infty$ we get
$$\left\|\Jacc^{\frac{1}{\bar p}}H_{k}^{\LH}(\Jacc^{\frac{1}{p^{*}}}u_{n}\circ \Cay,\eta\circ \Cay)\circ \Cay^{-1}\right\|_{H^{-k}}=o(1).$$
Now we move to estimating the term $\Jacc^{\frac{1}{p^{*}}-\frac{1}{\bar p}}u_{n} \LH(\eta\circ\Cay)\circ \Cay^{-1}$. Indeed, we have
$$\|\Jacc^{\frac{1}{p^{*}}}u_{n} \LH(\eta\circ\Cay)\circ \Cay^{-1}\|_{L^{\bar p}}\lesssim \|u_{n}\|_{L^{2}}\|\Jacc^{\frac{1}{p^{*}}}\LH(\eta\circ\Cay)\circ \Cay^{-1}\|_{L^{\frac{Q}{k}}}$$
and since $u_{n}\to 0$ in $L^{2}$ we have that
$$\|\Jacc^{\frac{1}{p^{*}}}u_{n} \LH(\eta\circ\Cay)\circ \Cay^{-1}\|_{L^{\bar p}}=o(1).$$
Therefore, we have that $\|\text{l.o.t}\|_{H^{-k}}=o(1)$.
By the sub-elliptic regularity estimates we find
\begin{gather}
\begin{split}
\label{eq: sobnorm eta un}
\SobNormB{\eta|u_n|}{k}{B_r(\zeta_0)}&\lesssim \SobNormB{\eta|u_n|^{p^*-2}u_n+\eta\delta_n}{-k}{B_r(\zeta_0)}+\|\text{l.o.t}\|_{H^{-k}(B_{r}(\zeta_{0}))}\\
&\lesssim  \SobNormB{\eta|u_n|^{p^*-2}u_n}{-k}{B_r(\zeta_0)}+\SobNormB{\eta\delta_n}{-k}{B_r(\zeta_0)}+o(1).
\end{split}
\end{gather}
We estimate the first term in the inequality above as follows
\begin{align*}
\SobNormB{\eta|u_n|^{p^*-2}u_n}{-k}{B_r(\zeta_0)}&\lesssim \LpNormB{\eta|u_n|^{p^*-2}u_n}{\frac{2Q}{Q+2k}}{B_r(\zeta_0)}\\
&\lesssim \LpNormB{u_n}{p^*}{B_r(\zeta_0)}^{\frac{4k}{Q-2k}}\LpNormB{\eta u_n}{p^*}{B_r(\zeta_0)}\\
&\lesssim\LpNormB{u_n}{p^*}{B_r(\zeta_0)}^{\frac{4k}{Q-2k}}\SobNormB{\eta u_n}{k}{B_r(\zeta_0)}.
\end{align*}
Substituting the estimates above in \eqref{eq: sobnorm eta un} and using \eqref{eq: int Br un <e}, we find
\begin{align*}
\SobNormB{\eta u_n}{k}{B_r(\zeta_0)}&\lesssim \LpNormB{u_n}{p^*}{B_r(\zeta_0)}^{\frac{4k}{Q-2k}}\SobNormB{\eta u_n}{k}{B_r(\zeta_0)}+o(1)\\
&\lesssim\e^{\frac{2k}{Q}}\SobNormB{\eta u_n}{k}{B_r(\zeta_0)}+o(1).
\end{align*}
Now, we choose $\e$ small enough to have $\SobNormB{\eta u_n}{k}{B_r(\zeta_0)}\to 0$, leading to a contradiction to our assumptions.
\end{proof}

\noindent
Given $r>0$, We can define now the concentration function
\begin{equation*}
Q_n(r)=\sup_{\zeta\in\Sph}\int_{B_r(\zeta)}|u_n|^{p^*}\dSph.
\end{equation*}
Since we are assuming that $u_n$ does not satisfies the (PS) condition, the Lemma above ensures the existence of a small enough $\e_0>0$ such that $\Sigma_{\e_0}\neq\emptyset $. Thus, for any fixed $\frac{\e_0}{3}>\e>0$, there exist a sequence of points $\zeta_n\in\Sph$ and a sequence of radii $R_n\to 0$ such that
\begin{equation}\label{eq: concentration fct = eps}
Q_n(R_n)=\int_{B_{R_n}(\zeta_n)}|u_n|^{p^*}\dSph=\e.
\end{equation}
Up to a subsequence, we can assume that $\zeta_n\to\zeta_0\in\Sph$ as $n\to \infty$. Again, we fix a coordinate system in $\C^{N+1}$ so that $\zeta_0=(1,0,\dots,0)$ and denote by $-\zeta_0=(-1,\dots,0,0)$ the antipodal point of $\zeta_0$. We set
\begin{equation*}
\Omega=\Cay^{-1}(B_1(\zeta_0))\subset\HN.
\end{equation*}
Clearly, for $n$ big all the balls $\Cay^{-1}(B_{R_n}(\zeta_n))$ will be contained in $\Omega$. Hence, by means of the map $\Cay^{-1}$, the problem of characterizing (PS) sequences can be studied in $\HN$, where the points $w_n=\Cay^{-1}(\zeta_n)$ accumulate at the point $0=\Cay^{-1}(\zeta_0)$ in the interior of the domain $\Omega$. Also, we define the map
\begin{equation*}
\rho_n:\HN\to \Sph\setminus\{-\zeta_0\},\qquad \rho_n(w)=\Cay\circ \tau_{w_n}\circ\delta_{R_n}(w)
\end{equation*}
and the functions
\begin{equation*}
U_n=\Jacr^{\frac{Q-2k}{2Q}}u_n\circ\rho_n,
\end{equation*}
here  $\Jacr$ is twice the absolute value of the Jacobian determinant of the map $\rho_n$. From now on we denote the preimage of a ball $B_{RR_n}(\zeta_n)\subset\Sph$ with respect to the function $\rho_n$ by
\begin{align*}
\set_R^n&=\rho_n^{-1}\left(B_{RR_n}(\zeta_n)\right).
\end{align*}
Notice that, for $n$ big, we can always assume $\zeta_n\in B_{\frac{1}{2}}(\zeta_0)$, hence $\set_R^n$ is well defined  and $\Cay^{-1}(B_{RR_n}(\zeta_n))\subset\Omega$ for every $RR_n<\frac{1}{2}$. Recalling the relation between $\LH$ and $\LSph$ expressed in \eqref{eq: relationship LH LSph}, we have
\begin{equation}\label{eq: int Up*=int up*}
\begin{split}
\int_{\set_R^n} U_n\LH U_n\dH&=\int_{B_{RR_n}(\zeta_n)} u_n\LSph u_n\dSph,\\
\int_{\set_R^n} |U_n|^{p^*}\dH&=\int_{B_{RR_n}(\zeta_n)} |u_n|^{p^*}\dSph.
\end{split}
\end{equation}
In the sequel we will make use of the following relation obtained from inclusions \eqref{eq: inclusions B and preimages}
\begin{equation*}
\set_R^n=\rho_n^{-1}(B_{R_nR}(\zeta_n))\supset \delta_{R_n^{-1}}\circ\tau_{w_n^{-1}}\left(B^{w_n}_{\frac{R_nR}{2}}\right)=B_{\frac{R}{2}}^0.
\end{equation*}
Also, we will use the notation
\begin{equation*}
\set_R^0=\Cay^{-1}(B_{R}(\zeta_0)).
\end{equation*}

\begin{lemma}\label{lem: Fn to 0}
Let us set $F_n=\LH U_n-|U_n|^{p^*-2}U_n$, then for every $R>0$
\begin{equation*}
\sup\left\{\la F_n,F\ra_{H^{-k},H^k}:\; \textnormal{supp}(F)\subset B_{R}^0,\; F\in H^k_0(B_R^0),\quad \SobNorm{F}{k}\leq 1\right\}\to 0
\end{equation*}
i.e.
\begin{equation*}
F_n\to 0\quad\text{in } \mlocSobH.
\end{equation*}
\end{lemma}
\begin{proof}
Let us consider $n$ big enough  to have $(6R_n)^{-1}\geq R$, and  $F\in H^k_0(B_R^0)$ such that $\textnormal{supp}(F)\subset B_R^0$ and $\SobNorm{F}{k}\leq 1$. We have
\begin{align*}
\la F_n,F\ra_{H^{-k},H^k}&=\int_{B_{(6R_n)^{-1}}^0}F\left(\LH U_n-|U_n|^{p^*-2}U_n\right)\dH\\
&\leq \int_{\set_{(3R_n)^{-1}}^n}F\left(\LH U_n-|U_n|^{p^*-2}U_n\right)\dH\\
&=\int_{\set_{(3R_n)^{-1}}^n}\Jacr^{\frac{Q+2k}{2Q}}F\left(\LSph u_n-|u_n|^{p^*-2}u_n\right)\circ\rho_n\dH\\
&=\int_{B_{\frac{1}{3}}(\zeta_n)}\left(\Jacr^{-\frac{Q-2k}{2Q}}F\right)\circ \sigma_n \left(\LSph u_n-|u_n|^{p^*-2}u_n\right)\dSph \;.
\end{align*}
On the other hand, recalling \eqref{eq: relationship LH LSph}, we find
\begin{align*}
\SobNorm{\left(\Jacr^{-\frac{Q-2k}{2Q}}F\right)\circ \sigma_n }{k}&=\int_{\Sph} \Jacr^{-1} (F\LH F)\circ \sigma_n \dSph\\
&=\int_{\HN} F\LH F\dH\leq C,
\end{align*}
thus $\la F_n,F\ra_{H^{-k},H^k}\to 0$.
\end{proof}

\begin{lemma}\label{lem: Un to Uinfty}
For $\e>0$ small enough in \eqref{eq: concentration fct = eps}, there exists $U_{\infty}\in D^k(\HN)$ such that $
U_n\to U_{\infty}$ in $\locSobH$ and
$$\LH U_{\infty}=|U_{\infty}|^{p^*-2}U_{\infty}\quad\text{on }\HN.$$
\end{lemma}
\begin{proof}
The sequence $U_n$ is bounded in $\locSobH$, hence there exists $U_{\infty}$ such that, up to subsequence, $U_n\rightharpoonup U_{\infty}$ weakly in in $\locSobH$ and $U_n\to U_{\infty}$ strongly in $L^p_{\textnormal{loc}}(\HN)$ for $1\leq p<p^*$. From \eqref{eq: int Up*=int up*} we deduce
\begin{equation}\label{eq: limsup int Uinfty <infty}
\limsup_{n\to \infty} \int_{\set_R^n}|U_n|^{p^*}\dH\leq \sup_{n\in\N}\int_{\Sph} |u_n|^{p^*}\dSph<\infty
\end{equation}
so that  $U_{\infty}\in L^{p^*}(\HN)$. Moreover, by the same argument given after the proof of Lemma \ref{lem: boundness of PS sequences}, we have that $U_{\infty}$ satisfies \eqref{eq: problem on H}, hence
\begin{equation}
\int_{\HN}U_{\infty}\LH U_{\infty}\dH<\infty.
\end{equation}
It follows that $U_{\infty}\in D^{k}(\HN)$. \\
In virtue of Lemma \ref{lem: the energy splits}, we replace $U_n$ by $U_n-U_{\infty}$ so that, from now to the end of the proof, we can assume $U_{\infty}=0$. By \eqref{eq: concentration fct = eps} we have
\begin{equation}\label{eq: concentration fct Hn=Sph=eps}
\int_{\set_{1}^n}|U_n|^{p^*}\dH=\int_{B_{R_n}(\zeta_n)}|u_n|^{p^*}\dSph=\e.
\end{equation}
Let $\beta\in C_0^{\infty}$, such that $\text{supp}(\beta)\subset \set_{1}^{n}$ then
\begin{equation}\label{eq: estimate beta2Un}
\begin{split}
\SobNorm{\beta U_n}{k}&\lesssim \SobNorm{\LH(\beta U_n)}{-k}+\LpNorm{\beta U_n}{2}\;.
\end{split}
\end{equation}
Again, we use the fact that
\begin{align*}
\LH(\beta U_n)&=\beta \LH(U_{n})+U_{n}\LH(\beta )+H^{\LH}(U_{n},\beta )\notag\\
&=\beta \LH(U_{n})+\text{l.o.t.}
\end{align*}
So first, we have that
$$U_{n}\LH(\beta)=U_{n}\chi_{B_{R}^0}\LH(\beta )+U_{n}\chi_{\mathbb{H}^{N}\setminus B_{R}^0}\LH(\beta ) \;.$$
Therefore,
\begin{align*}
\|U_{n}\LH(\beta )\|_{L^{\bar p}}&\lesssim \|U_{n}\|_{L^{2}(B_{R}^0)}\|\LH(\beta )\|_{L^{\frac{Q}{k}}}+\|U_{n}\|_{L^{p^{*}}}\|\chi_{\mathbb{H}^{N}\setminus B_{R}^0}\LH(\beta )\|_{L^{\frac{Q}{2k}}}\notag\\
&\lesssim \|U_{n}\|_{L^{2}(B_{R}^0)}+\frac{1}{R^{Q}}\|U_{n}\|_{L^{p^{*}}} \;.
\end{align*}
Since $U_{n}\to 0$ in $L^{2}_{loc}$ and $\|U_{n}\|_{L^{p^{*}}}$ is bounded, if we let $n\to \infty$ and then $R\to \infty$, we have that $\|U_{n}\LH(\beta )\|_{L^{\bar p}}=o(1)$. Next, we move to the term $H^{\LH}(\beta ,U_{n})$. Again, we have that from Lemma (commutator estimates),
$$|H^{\LH}(\beta ,U_{n})|(x)\lesssim \sum_{j=1}^{L}R_{2k-s_{j,1}-s_{j,2}}\Big(R_{s_{j,1}}|\LH (U_{n})|R_{s_{j,2}}|\LH(\beta )|\Big)(x) \;.$$
So we consider one term of the form $R_{2k-s-t}\Big(R_{t}|\LH(U_{n})|R_{s}|\LH(\beta )|\Big)(x)$. Using the same splitting as in \eqref{eq: splitting}, we have that
$$\|H^{\LH}(\beta ,U_{n})\|_{H^{-k}}=o(1),$$
and thus $\|\text{l.o.t}\|_{H^{-k}}=o(1)$. Clearly $\LpNorm{\beta  U_n}{2}\to 0$, and by Lemma \ref{lem: Fn to 0} we know that $F_n\to 0$ in $\mlocSobH$, hence, we have
\begin{align*}
\SobNorm{\LH(\beta U_n)}{-k}&\leq \SobNorm{\beta \LH U_n+\text{l.o.t.}}{-k}\\
&\leq \SobNorm{\beta \left(|U_n|^{p^*-2}U_n+F_n\right)}{-k}+o(1)\\
&\leq  \SobNorm{\beta |U_n|^{p^*-2}U_n}{-k}+o(1).
\end{align*}
Therefore
\begin{equation*}
\SobNorm{\beta  U_n}{k}\lesssim\SobNorm{\beta|U_n|^{p^*-2}U_n}{-k}+o(1)
\end{equation*}
and

\begin{align*}
\SobNorm{\beta  U_n}{k}&\lesssim\left\| \beta |U_n|^{p^*-2}U_n \right\|_{L^{\bar{p}}(\set_1^n)}+o(1)\\
&\lesssim\left(\int_{\set_1^{n}}|U_n|^{p^*}\dH\right)^{\frac{2k}{Q}}\LpNormB{\beta U_n}{p^*}{\set_1^n}+o(1)\\
&\lesssim\epsilon^{\frac{2k}{Q}}\LpNormB{\beta U_n}{p^*}{\set_1^n}+o(1)\\
&=o(1)
\end{align*}
as $n\to\infty.$
\end{proof}

\noindent
From the Lemma above and \eqref{eq: concentration fct Hn=Sph=eps}, it follows
\begin{equation*}
\int_{\set_1^0} |U_{\infty}|^{p^*}\dH=\e,
\end{equation*}
hence $U_{\infty}\neq 0$ is a solution to \eqref{eq: problem on H}. We consider a cut off function $\gamma$ such that $\gamma\equiv 1$ on $B_{\frac{1}{4}}^0$, $\textnormal{supp}(\gamma)\subset B_{\frac{1}{2}}^0, $ and we define $\beta=\gamma \circ \Cay^{-1}$. In virtue of the inclusions
\begin{equation*}
\Cay^{-1}(B_{\frac{1}{4}}(\zeta_0))\subseteq B_{\frac{1}{4}}^0\subseteq \Cay^{-1}(B_{\frac{1}{2}}(\zeta_0))\subseteq  B_{\frac{1}{2}}^0,
\end{equation*}
the function $\beta$ is a cut off function such that $\beta\equiv 1$ on $B_{\frac{1}{4}}(\zeta_0)$ and $\textnormal{supp}(\beta)\subset B_1(\zeta_0),$ moreover for $n$ big enough we have
\begin{gather*}
\textnormal{supp}(\beta\circ\rho_n)=\textnormal{supp}(\gamma\circ\tau_{w_n}\circ\delta_{R_n})\subseteq B_{(2R_n)^{-1}}^{w_n}\subseteq B_{R_n^{-1}}^0,\\
\beta\circ\rho_n\equiv 1\quad\text{on } B_{(6R_n)^{-1}}^0.
\end{gather*}
We set
\begin{equation}\label{eq: v_n}
v_n=\Jacs^{\frac{Q-2k}{2Q}}\beta U_{\infty}\circ\sigma_n 
\end{equation}
where $\Jacs$ is half the absolute value of the Jacobian determinant of $\sigma_n$, and  consider
$$\ol{u}_n=u_n-v_n\;. $$
For clarity sake, we recall here the definition of $u_n$ with respect to  $U_n$
\begin{equation*}
u_n=\Jacs^{\frac{Q-2k}{2Q}}U_n\circ\sigma_n.
\end{equation*}
We have then
\begin{lemma}
After taking a subsequence if necessary, we have
\begin{equation*}
\ol{u}_n\rightharpoonup 0\quad\text{weakly in }\SobSph.
\end{equation*}
\end{lemma}
\begin{proof}
Since we have already proved that $u_n\rightharpoonup 0$, it suffices to show that $v_n\rightharpoonup 0$ weakly in $\SobSph$. On the other hand, $v_n$ is bounded in $\SobSph$ so that, after taking a subsequence if necessary, it converges to some limit; hence we only need to prove that the distributional limit is zero, i.e. it suffices to prove that for $f\in C^{\infty}$ it holds
\begin{equation*}
\int_{\Sph} v_n f\dSph\to 0.
\end{equation*}
Let us fix $R>0$. We estimate the integral above, first on $B_{R_nR}(\zeta_n)$ and then on the exterior domain $\Sph\setminus B_{R_nR}(\zeta_n)$, we have
\begin{align*}
\left| \int_{ B_{R_nR}(\zeta_n)}v_nf \dSph\right|&=\left|\int_{ B_{R_nR}(\zeta_n)} \Jacs^{\frac{Q-2k}{2Q}}f\beta U_{\infty}\circ\sigma_n\dSph \right|\\
&=\left|\int_{\set_{R}^n} \Jacr^{\frac{Q+2k}{2Q}} U_{\infty}(f\beta)\circ\rho_n\dH \right|\\
&\leq CR_n^{\frac{Q+2k}{2}}\|f\|_{\infty} \|\Jacc\|_{\infty}^{\frac{Q+2k}{2Q}}\int_{\set_{R}^n} |U_{\infty}|\dH.
\end{align*}
On the exterior domain, for $n$ big enough we find
\begin{align*}
\left| \int_{\Sph\setminus B_{R_nR}(\zeta_n)}v_nf \dSph\right|&=\left|\int_{ B_1(\zeta_0)\setminus B_{R_nR}(\zeta_n)} \Jacs^{\frac{Q-2k}{2Q}}f\beta U_{\infty}\circ\sigma_n\dSph \right|\\
&\leq\left|\int_{B_{2R_n^{-1}}^n\setminus \set_{R}^n} \Jacr^{\frac{Q+2k}{2Q}} U_{\infty}(f\beta)\circ\rho_n\dH \right|\\
&\leq CR_n^{\frac{Q+2k}{2}}\|f\|_{\infty} \|\Jacc\|_{\infty}^{\frac{Q+2k}{2Q}}\int_{ B_{2R_n^{-1}}^0\setminus B_{\frac{R}{2}}^0} |U_{\infty}|\dH.\\
&\leq C\|f\|_{\infty} \|\Jacc\|_{\infty}^{\frac{Q+2k}{2Q}}\left(\int_{ B_{2R_n^{-1}}^0\setminus B_{\frac{R}{2}}^0} |U_{\infty}|^{p*}\dH\right)^{\frac{1}{p^*}}.
\end{align*}
Then, the thesis follow letting $n\to \infty$ and then $R\to \infty$ in the following estimate
\begin{align}
\left|\int_{\Sph} v_n f\dSph\right| &\lesssim \|f\|_{\infty} \|\Jacc\|_{\infty}^{\frac{Q+2k}{2Q}}R_n^{\frac{Q+2k}{2}}\int_{ \set_{R}^n} |U_{\infty}|\dH\notag \\
&\quad+\|f\|_{\infty} \|\Jacc\|_{\infty}^{\frac{Q+2k}{2Q}} \LpNormB{U_{\infty}}{p^*}{B_{2R_n^{-1}}^0\setminus B_{\frac{R}{2}}^0}.\notag
\end{align}
\end{proof}

\begin{lemma}\label{lem: dE(vn) to 0}
We have
\begin{equation*}
dE(v_n)\to 0 \text{ in }\mSobSph\qquad\text{and}\qquad dE(\ol{u}_n)\to 0\text{ in }\mSobSph.
\end{equation*}
\end{lemma}
\begin{proof}
Let $f\in\SobSph$ and $f_n=\LSph v_n-|v_n|^{p^*-2}v_n$. First we notice that
\begin{align}
\LSph(v_{n})&=\left(\Jacr^{-\frac{1}{\bar p}}\LH(\Jacr^{\frac{1}{p^{*}}}v_{n}\circ\rho_{n})\right)\circ \sigma_{n}\notag \\
&=\left(\Jacr^{-\frac{1}{\bar p}}\LH(\beta\circ\rho_{n}U_{\infty})\right)\circ \sigma_{n} \;.
\end{align}
Thus, we get,
\begin{align*}
\int_{\Sph}f_nf \dSph&= \int_{\Sph} f\Big (\LSph(v_{n})-|v_{n}|^{p^{*}-2}v_{n}\Big) \dSph \notag\\
&=\int_{\HN}\Jacr^{\frac{1}{p^{*}}}f\circ \rho_{n}\LH(\beta\circ \rho_{n} U_{\infty}) \dH-\int_{\Sph}f|v_{n}|^{p^{*}-2}v_{n}\dSph. \notag
\end{align*}
Now notice that
\begin{align*}
\int_{\HN}\Jacr^{\frac{1}{p^{*}}}f\circ \rho_{n}\LH(\beta\circ \rho_{n} U_{\infty}) \dH &=\int_{\HN}\Jacr^{\frac{1}{p^{*}}}f\circ \rho_{n}U_{\infty}\LH(\beta\circ \rho_{n})\dH\\
&\quad+\int_{\HN}\Jacr^{\frac{1}{p^{*}}}(f\beta)\circ \rho_{n} \LH(U_{\infty})\dH\\
&\quad+\int_{\HN}\Jacr^{\frac{1}{p^{*}}}f\circ \rho_{n} H^{\LH}(U_{\infty},\beta\circ \rho_{n})\dH\notag\\
&=\int_{\Sph}f\beta (\Jacr^{-\frac{1}{\bar p}}|U_{\infty}|^{p^{*}-2}U_{\infty})\circ \sigma_{n}\dSph\\
&\quad+\int_{\HN}\Jacr^{\frac{1}{p^{*}}}f\circ \rho_{n}U_{\infty}\LH(\beta\circ \rho_{n})\dH \notag\\
&\quad + \int_{\HN}\Jacr^{\frac{1}{p^{*}}}f\circ \rho_{n} H^{\LH}(U_{\infty},\beta\circ \rho_{n}) \dH.
\end{align*}
Therefore, we have that
\begin{align*}
\int_{\Sph}f_nf \dSph&=\int_{\HN}\Jacr^{\frac{1}{p^{*}}}f\circ \rho_{n}U_{\infty}\LH(\beta\circ \rho_{n})\dH \notag\\
&\qquad +\int_{\HN}\Jacr^{\frac{1}{p^{*}}}f\circ \rho_{n} H^{\LH}(U_{\infty},\beta\circ \rho_{n}) \dH \notag\\
&\qquad + \int_{\Sph}f\Big(\beta-\beta^{p^{*}-1}\Big) \Big(\Jacr^{-\frac{1}{\bar p}}|U_{\infty}|^{p^{*}-2}U_{\infty}\Big)\circ \sigma_{n}\dSph \notag\\
&=I_{1}+I_{2}+I_{3}.
\end{align*}
We estimate each of the three terms above separately. But first, we notice that
$$\|\LH(\beta\circ \rho_{n})\|_{L^{p}}=R_{n}^{2k-\frac{Q}{p}}\|\LH(\gamma)\|_{L^{p}}.$$
In particular, if $p>\frac{Q}{2k}$, then $\|\LH(\beta\circ \rho_{n})\|_{L^{p}}\to 0$. Now we have, for $R>1$,
\begin{align*}
|I_1|&=\left|\int_{\HN}\Jacr^{\frac{1}{p^{*}}}f\circ \rho_{n}U_{\infty}\LH(\beta\circ \rho_{n})\dH\right|\notag \\
&\leq \|\Jacr^{\frac{1}{p^{*}}}f\circ \rho_{n}\|_{L^{p^{*}}}\|U_{\infty}\LH(\beta\circ \rho_{n})\|_{L^{\bar p}}\notag \\
&\leq \|f\|_{H^{k}}\Big(\|U_{\infty}\|_{L^{q}}\|\LH(\beta\circ \rho_{n})\|_{L^{p}(B_{R}^0)}+\|U_{\infty}\|_{L^{p^{*}}(\HN\setminus B_{R}^0)}\|\LH(\beta\circ \rho_{n})\|_{L^{\frac{Q}{2k}}}\Big) \;,
\end{align*}
where $\frac{1}{p}+\frac{1}{q}=\frac{1}{\bar p}$. Taking $p>\frac{Q}{2k}$, we have for $R$ fixed that
$$\|U_{\infty}\|_{L^{q}(B_{R}^0)}\|\LH(\beta\circ \rho_{n})\|_{L^{p}}\to 0 \text{ as } n\to \infty \;.$$
On the other hand, we have that
$$\|U_{\infty}\|_{L^{p^{*}}(\HN\setminus B_{R}^0)}\|\LH(\beta\circ \rho_{n})\|_{L^{\frac{Q}{2k}}}=\|U_{\infty}\|_{L^{p^{*}}(\HN\setminus B_{R}^0)}\|\LH(\gamma)\|_{L^{\frac{Q}{2k}}}\to 0 \text{ as } R\to \infty \;.$$
Hence
$$|I_{1}|=o(1)\|f\|_{H^{k}}.$$
We move now to the term $I_{2}$. First, we recall the following estimate for the Riesz potentials:
\begin{equation*}\label{Ri}
\|R_{2k-s-t}(R_{t}(u)R_{s}(v))\|_{L^{p}}\lesssim \|u\|_{L^{q_{1}}}\|v\|_{L^{q_{2}}},
\end{equation*}
for $\frac{1}{p}=\frac{1}{q_{1}}+\frac{1}{q_{2}}-\frac{2k}{Q}$. In particular we have from Lemma (commutator estimates) and the previous estimates,
\begin{align*}
\|H^{\LH}(U_{\infty},\beta\circ \rho_{n})\|_{L^{\bar p}}&\lesssim \|U_{\infty}^{p^{*}-1}\|_{L^{q}(B_{R}^0)}\|\LH(\beta\circ \rho_{n})\|_{L^{p}}\\
&\qquad+\|U_{\infty}^{p^{*}-1}\|_{L^{\bar p}(\HN\setminus B_{R}^0)}\|\LH(\beta\circ \rho_{n})\|_{L^{\frac{Q}{2k}}}\notag\\
& \lesssim \|U_{\infty}\|_{L^{q(p^{*}-1)}(B_{R}^0)}^{p^{*}-1}\|\LH(\beta\circ \rho_{n})\|_{L^{p}}\\
&\qquad+\|U_{\infty}\|_{L^{p^{*}}(\HN\setminus B_{R}^0)}^{p^{*}-1}\|\LH(\gamma)\|_{L^{\frac{Q}{2k}}}
\end{align*}
where $\frac{1}{\bar p}=\frac{1}{p}+\frac{1}{q}-\frac{2k}{Q}$. Hence, taking $p>\frac{2k}{Q}$ and letting first $n\to 0$ then $R\to \infty$, we have that
\begin{equation}\label{esth}
\|H^{\LH}(U_{\infty},\beta\circ \rho_{n})\|_{L^{\bar p}}=o(1).
\end{equation}
In particular,
$$
|I_2|\leq \|f\|_{H^{k}}\|H^{\LH}(U_{\infty},\beta\circ \rho_{n})\|_{L^{\bar p}}=o(1)\|f\|_{H^k}\quad \text{as } n\to\infty.
$$
Now we estimate the term $I_3$.
\begin{align*}
|I_3|\lesssim\LpNormB{(\beta-\beta^{p^*-1})\circ\rho_n|U_{\infty}|^{p^*-2}U_{\infty}}{\bar{p}}{\HN}\|f\|_{H^k}
\end{align*}
but $U_{\infty}\in D^k(\HN)$ and
\begin{align*}
\LpNormB{(\beta-\beta^{p^*-1})\circ\rho_n|U_{\infty}|^{p^*-2}U_{\infty}}{\bar{p}}{\HN}\leq C\LpNormB{U_{\infty}}{p^*}{B_{2R_{n}^{-1}}^0\setminus B_{(8R_n)^{-1}}^0}^{\frac{Q+2k}{Q-2k}}
\end{align*}
so that
\begin{equation*}
|I_3|\leq o(1)\SobNorm{f}{k}.
\end{equation*}
Hence we have proved that $f_n\to 0$ in $\mSobSph$. Now we turn to $dE(\ol{u}_n)$. Again, we consider $f\in \SobSph$ and compute
\begin{align*}
\la dE(\ol{u}_n),f\ra&=\la dE(u_n),f\ra-\la dE(v_n),f\ra \\
&\quad+\int_{\Sph}\left(|u_n|^{p^*-2}u_n-|v_n|^{p^*-2}v_n-|\ol{u}_n|^{p^*-2}\ol{u}_n\right)f\dSph.
\end{align*}
We notice that, since $dE(u_n)$ and $dE(v_n)$ converge to zero in $\mSobSph$, it suffices to show
\begin{equation}\label{eq: An to 0 in H-k}
A_n=|u_n|^{p^*-2}u_n-|v_n|^{p^*-2}v_n-|\ol{u}_n|^{p^*-2}\ol{u}_n\to 0\quad\text{ in } \mSobSph.
\end{equation}
In order to prove \eqref{eq: An to 0 in H-k}, we will show $\LpNormB{A_n}{\bar{p}}{\Sph}\to 0$. Let us fix $R>0$. First we want to and obtain an estimate for $A_n$ in the exterior domain $D_n=\Sph\setminus B_{RR_n}(\zeta_n)$ and then we will move to the interior of the ball $B_{RR_n}(\zeta_n)$. We write $u_n=\ol{u}_n+v_n$ in the definition of $A_n$ and we notice that for a big enough positive constant $C$ we have
\begin{align*}
|A_n|&=\left||\ol{u}_{n}+v_n|^{p^*-2}(\ol{u}_{n}+v_n)-|\ol{u}_n|^{p^*-2}\ol{u}_n-|v_n|^{p^*-2}v_n\right|\\
&\leq C\left( |\ol{u}_{n}|^{p^*-2}|v_n|+ |v_n|^{p^*-2}|\ol{u}_n|\right).
\end{align*}
Hence, by the H\"older inequality and recalling that $\textnormal{supp}(\beta\circ\rho_n)\subseteq B_{2R_n^{-1}}^0$ and $B_{\frac{R}{2}}^0\subseteq \set_R^n$, we have
\begin{align*}
\LpNormB{A_n}{\bar{p}}{D_n}&\lesssim \left(\LpNormB{|\ol{u}_n|^{p^*-2}v_n}{\bar{p}}{D_n}+\LpNormB{\ol{u}_n|v_n|^{p^*-2}}{\bar{p}}{D_n}\right)\\
&\lesssim \left(\LpNormB{\ol{u}_n}{p^*}{D_n}^{\frac{4k}{Q-2k}}\LpNormB{v_n}{p^*}{D_n}+\LpNormB{\ol{u}_n}{p^*}{D_n}\LpNormB{v_n}{p^*}{D_n}^{\frac{4k}{Q-2k}}\right)\\
&\lesssim \LpNormB{\ol{u}_n}{p^*}{\Sph}^{\frac{4k}{Q-2k}}\LpNormB{U_{\infty}}{p^*}{B_{2R_n^{-1}}^0\setminus B_{\frac{R}{2}}^0}\\
&\quad+\LpNormB{\ol{u}_n}{p^*}{\Sph}\LpNormB{U_{\infty}}{p^*}{B_{2R_n^{-1}}^0\setminus B_{\frac{R}{2}}^0}^{\frac{4k}{Q-2k}}\\
&=o(1),
\end{align*}
as $R\to\infty$, uniformly in $n$. Similarly, in the interior of the ball $B_{RR_n}(\zeta_n)$ we find
\begin{align*}
\LpNormB{A_n}{\bar p}{B_{RR_n}(\zeta_n)}&\lesssim \left(\LpNormB{|\ol{u}_n|^{p^*-2}v_n}{\bar p}{B_{RR_n}(\zeta_n)}+\LpNormB{\ol{u}_n|v_n|^{p^*-2}}{\bar p}{B_{RR_n}(\zeta_n)}\right)\\
&\lesssim \LpNormB{\ol{u}_n}{p^*}{B_{RR_n}(\zeta_n)}^{\frac{4k}{Q-2k}}\LpNormB{v_n}{p^*}{B_{RR_n}(\zeta_n)}\\
&\quad+\LpNormB{\ol{u}_n}{p^*}{B_{RR_n}(\zeta_n)}\LpNormB{v_n}{p^*}{B_{RR_n}(\zeta_n)}^{\frac{4k}{Q-2k}}\\
&\lesssim \LpNormB{U_n-(\beta\circ\rho_n)U_{\infty}}{p^*}{\set_R^n}^{\frac{4k}{Q-2k}}\LpNormB{(\beta\circ\rho_n)U_{\infty}}{p^*}{\set_R^n}\\
&\quad+\LpNormB{U_n-(\beta\circ\rho_n)U_{\infty}}{p^*}{\set_R^n}\LpNormB{(\beta\circ\rho_n)U_{\infty}}{p^*}{\set_R^n}^{\frac{4k}{Q-2k}}.
\end{align*}
Therefore, recalling the fact that $U_n\to U_{\infty}$ in $\locSobH$, we finally obtain
\begin{equation*}
\int_{B_{RR_n}(\zeta_n)}|A_n|^{\bar p}\dSph \to 0
\end{equation*}
as desired.
\end{proof}

\begin{lemma}\label{lem: the energy splits2}
We have the following energy estimate
\begin{equation*}
E(\ol{u}_n)=E(u_n)-E_{\H}(U_{\infty})+o(1).
\end{equation*}
\end{lemma}
\begin{proof}
We compute
\begin{gather}\label{eq: E(ol(un))}
\begin{split}
E(\ol{u}_n)&=\frac{1}{2}\left(\int_{\Sph}u_n\LSph u_n\dSph+\int_{\Sph}v_n\LSph v_n\dSph-2\int_{\Sph}u_n\LSph v_n\dSph\right)\\
&\qquad-\frac{1}{p^*}\int_{\Sph}|\ol{u}_n|^{p^*}\dSph.\\
\end{split}
\end{gather}
Let us estimate the second term in the identity above
\begin{align*}
\int_{\Sph} v_n\LSph v_n\dSph&=\int_{\HN} \beta\circ\rho_n U_{\infty}\LH (\beta\circ\rho_n  U_{\infty})\dH\\
&=\int_{\HN} \beta^2\circ\rho_n U_{\infty}\LH  U_{\infty}\dH+\int_{\HN} \beta\circ\rho_n U_{\infty}^2\LH (\beta\circ\rho_n)\dH\\
&\quad+\int_{\HN}\beta\circ\rho_{n}U_{\infty}H^{\LH}(U_{\infty},\beta\circ\rho_{n})\dH\\
&=I_1+I_2+I_3.
\end{align*}
Clearly,
\begin{equation*}
I_1=\int_{\HN}  U_{\infty}\LH  U_{\infty}\dH+o(1)
\end{equation*}
and using the same argument as in Lemma \ref{lem: Un to Uinfty},
\begin{gather*}
\begin{split}
\left|I_2\right|&=\left|\int_{\HN} \beta\circ\rho_n U_{\infty}^2\LH (\beta\circ\rho_n)\dH \right|\\
&\lesssim \|U_{\infty}^{2}\|_{L^{q}(B_{R}^0)}\|\LH (\beta\circ\rho_n)\|_{L^{p}}+\|U_{\infty}^{2}\|_{L^{\frac{Q}{Q-2k}}(\HN\setminus B_{R}^0)}\|\LH (\beta\circ\rho_n)\|_{L^{\frac{Q}{2k}}}\notag\\
&\lesssim \|U_{\infty}\|^{2}_{L^{2q}(B_{R}^0)}\|\LH (\beta\circ\rho_n)\|_{L^{p}}+\|U_{\infty}\|^{2}_{L^{p^{*}}(\HN\setminus B_{R}^0)}\|\LH (\gamma)\|_{L^{\frac{Q}{2k}}}
\end{split}
\end{gather*}
for $\frac{1}{p}+\frac{1}{q}=1$, taking $p>\frac{Q}{2k}$ and letting $n\to \infty$ then $R\to \infty$, we have that
$$I_{2}=o(1),$$
as $n\to\infty.$ Now for $I_{3}$, using $(\ref{esth})$, we have
$$|I_{3}|\lesssim \|U_{\infty}\|_{L^{p^{*}}}\|H^{\LH}(U_{\infty},\beta\circ\rho_{n})\|_{L^{\bar p}}=o(1).$$
Combining the estimates of $I_1$, $I_2$ and $I_3$, we find
\begin{equation}\label{eq: estimate vnLvn}
\int_{\Sph}v_n\LSph v_n\dSph=\int_{\HN} U_{\infty}\LH U_{\infty} \dH +o(1) \;.
\end{equation}
Now we estimate the third term in \eqref{eq: E(ol(un))}
\begin{align*}
\int_{\Sph} u_n\LSph v_n\dSph&=\int_{\HN} U_n \LH (\beta\circ\rho_{n}U_{\infty})\dH\\
&=\int_{\HN} \beta\circ\rho_{n} U_n \LH U_{\infty}\dH+\int_{\HN} U_n U_{\infty}\LH (\beta\circ\rho_{n})\dH\\
&\quad + \int_{\HN}U_{n}H^{\LH}(U_{\infty},\beta\circ \rho_{n})\dH\\
&=I_4+I_5+I_6.
\end{align*}
Let us fix $R>0$ and define $B_n=B_{2R_n^{-1}}^0\setminus B_R^0$. For $n$ big enough to have $\beta\circ\rho_n\equiv 1$ on $B_{(6R_n)^{-1}}^0\supset B_R^0 $, we get
\begin{align*}
I_4=\int_{B_R^0}U_n\LH U_{\infty}\dH+\int_{B_n}\beta\circ\rho_n U_n\LH U_{\infty}\dH \;,
\end{align*}
and we estimate the second term in the identity above by
\begin{align*}
\left| \int_{B_n}\beta\circ\rho_n U_n\LH U_{\infty}\dH \right| &\leq C\LpNormB{U_n}{p^*}{B_{2R_n^{-1}}^0}\LpNormB{\LH U_{\infty}}{\bar{p}}{B_n}=o(1) \;,
\end{align*}
as $R\to \infty$ uniformly in $n$. Hence, since $U_n\to U_{\infty}$ in $\locSobH$
\begin{equation*}
I_4=\int_{B_R^0}U_{\infty}\LH U_{\infty}\dH+ o(1).
\end{equation*}
Let us turn the attention to $I_5$
\begin{align*}
|I_5|&=\left|\int_{\HN} U_n U_{\infty}\LH (\beta\circ\rho_{n})\dH\right|\\
&\lesssim \|U_{n}\|_{L^{p^{*}}}\Big(\|U_{\infty}\|_{L^{q}(B_{R})}\|\LH (\beta\circ\rho_n)\|_{L^{p}}+\|U_{\infty}\|_{L^{p^{*}}(\HN\setminus B_{R}^0)}\|\LH (\gamma)\|_{L^{\frac{Q}{2k}}}\Big) \;,
\end{align*}
where $\frac{1}{p}+\frac{1}{q}=\frac{1}{\bar p}$. Once again, if we take $p>\frac{Q}{2k}$, and let $n\to \infty$ then $R\to \infty$ we get
$$I_{5}=o(1).$$
Also,
\begin{align*}
|I_6|&\leq \|U_{n}\|_{L^{p^{*}}}\|H^{\LH}(U_{\infty},\beta\circ \rho_{n})\|_{L^{\bar p}}=o(1) \;.
\end{align*}
Combining the estimates for $I_4$, $I_5$ and $I_6$ we get
\begin{equation}\label{eq: estimate unLvn}
\int_{\Sph} u_n\LSph v_n \dSph =\int_{B_R^0}U_{\infty}\LH U_{\infty}\dH +o(1) \;.
\end{equation}
We consider now the last term in \eqref{eq: E(ol(un))}. We are going to show that
\begin{equation}\label{eq: estimate olun}
\int_{\Sph} |\ol{u}_n|^{p^*} \dSph= \int_{\Sph} |u_n|^{p^*} \dSph-\int_{\HN} |U_{\infty}|^{p^*} \dH+o(1)\quad\text{as } n\to\infty.
\end{equation}
Hence, using \eqref{eq: estimate vnLvn}, \eqref{eq: estimate unLvn} and \eqref{eq: estimate olun} to estimate the right hand side of \eqref{eq: E(ol(un))} we get
\begin{equation*}
E(\ol{u}_n)=E(u_n)-E_{\H}(U_{\infty})+o(1) \;,
\end{equation*}
as desired. Before proving \eqref{eq: estimate olun} we make a few observations. Let us fix $R>0$ and define $D_n=\Sph\setminus B_{RR_n}(\zeta_n)$. First, we notice that for $n$ big enough to have $\beta\circ\rho_n\equiv 1$ on $B_R^0$ we find
\begin{equation}\label{eq: estimate olun on B}
\int_{B_{RR_n}(\zeta_n)}|\ol{u}_n|^{p^*}\dSph=\int_{\set_R^n}|U_n-U_{\infty}|^{p^*}\dH=o(1) \;,
\end{equation}
as $ n\to \infty$, since by Lemma \ref{lem: Un to Uinfty} $U_n\to U_{\infty}$ in $\locSobH$. Also,
\begin{align}\label{eq: estimate vn on Dn}
\int_{D_n}|v_n|^{p^*}\dSph=\int_{B_{2R_n^{-1}}^0\setminus \set_R^n}|\beta\circ\rho_n U_{\infty}|^{p^*}\dH\leq C\int_{\HN\setminus B_{\frac{R}{2}}^0}| U_{\infty}|^{p^*}\dH=o(1) \;,
\end{align}
as $R\to\infty$ uniformly in $n$. Moreover
\begin{equation}\label{eq: estimate olun on Dn}
\int_{D_n}|\ol{u}_n|^{p^*}\dSph=\int_{D_n}|u_n|^{p^*}\dSph-\int_{D_n}|v_n|^{p^*}\dSph+o(1) \;,
\end{equation}
as $R\to \infty$ uniformly in $n$. Indeed, for a suitable constant $C$, independent of $n$ we have
\begin{equation*}
\int_{D_n}\left||u_n|^{p^*}-|\ol{u}_n|^{p^*}-|v_n|^{p^*}\right|\dSph\leq C\int_{D_n}|\ol{u}_n|^{p^*-1}|v_n|\dSph+C\int_{D_n}|v_n|^{p^*-1}|\ol{u}_n|\dSph \;,
\end{equation*}
but
\begin{align*}
\int_{D_n}|\ol{u}_n|^{p^*-1}|v_n|\dSph&=\int_{\HN\setminus \set_R^n}|U_n-\beta\circ\rho_n U_{\infty}|^{p^*-1}|\beta\circ\rho_n U_{\infty}|\dH\\
&\lesssim \int_{B_n}|U_n-\beta\circ\rho_n U_{\infty}|^{p^*-1}| U_{\infty}|\dH\\
&\lesssim \LpNormB{U_n-\beta\circ\rho_n U_{\infty}}{p^*}{B_n}^{\frac{Q+2k}{Q-2k}}\LpNormB{U_{\infty}}{p^*}{B_n}\\
&\lesssim \LpNormB{U_{\infty}}{p^*}{B_n}\\
&=o(1) \;,
\end{align*}
as $R\to \infty$ uniformly in $n$. Similarly
\begin{align*}
\int_{D_n}|v_n|^{p^*-1}|\ol{u}_n|\dSph&\lesssim \int_{\HN\setminus B_R^0}|\beta\circ\rho_n U_{\infty}|^{p^*-1}|U_n-\beta\circ\rho_n U_{\infty}|\dH\\
&\lesssim \LpNormB{U_{\infty}}{p^*}{B_n}^{\frac{Q+2k}{Q-2k}}\LpNormB{U_n-\beta\circ\rho_nU_{\infty}}{p^*}{B_n}\\
&\lesssim \LpNormB{U_{\infty}}{p^*}{B_n}\\
&=o(1) \;,
\end{align*}
as $R\to \infty$ uniformly in $n$, proving \eqref{eq: estimate olun on Dn}. We are ready now to prove \eqref{eq: estimate olun}:
\begin{align*}
\int_{\Sph}|\ol{u}_n|^{p^*}\dSph&=\int_{B_{RR_n}(\zeta_n)}|\ol{u}_n|^{p^*}\dSph+\int_{D_n}|\ol{u}_n|^{p^*}\dSph\\
&\overset{\eqref{eq: estimate olun on B}}{=} \int_{D_n}|\ol{u}_n|^{p^*}\dSph+o(1)\\
&\overset{\eqref{eq: estimate olun on Dn}}{=} \int_{D_n}|u_n|^{p^*}\dSph-\int_{D_n}|v_n|^{p^*}\dSph+o(1)\\
&\overset{\eqref{eq: estimate vn on Dn}}{=} \int_{D_n}|u_n|^{p^*}\dSph+o(1)\\
&=\int_{\Sph}|u_n|^{p^*}\dSph-\int_{B_{RR_n}(\zeta_n)}|u_n|^{p^*}\dSph+o(1)\\
&=\int_{\Sph}|u_n|^{p^*}\dSph-\int_{\set_R^n}|U_n|^{p^*}\dH+o(1).
\end{align*}
Now, recalling that $\int_{\set_R^n}|U_n|^{p^*}\dH\to\int_{\set_R^0}|U_{\infty}|^{p^*}\dH$ for any $R$ as $n\to \infty$ we finally get \eqref{eq: estimate olun}.
\end{proof}

\begin{remark}\label{rmk: Energy of the bubbles}
Let $\omega$ be defined by \eqref{eq: optimizer function}. Let us explicitly recall that the functions
\begin{equation*}
\omega_{\lambda,\xi}=\lambda^{\frac{2k-Q}{2}}\omega\circ\delta_{\lambda^{-1}}\circ\tau_{\xi^{-1}},\quad \lambda>0,\; \xi\in\HN\;,
\end{equation*}
are solutions to \eqref{eq: problem on H} which have all the same energy
\begin{equation*}
C_E:=E_{\H}(\omega_{\lambda,\xi})=\frac{k}{Q} C_{S}^{-\frac{Q}{2k}}>0, \quad \forall \lambda>0 \text{ and } \xi\in\HN.
\end{equation*}
where $C_S$ is the Sobolev constant in (\ref{eq: Sobolev constant Sph}). In particular they are the only ones with this energy (\cite{FL}).
\end{remark}


\noindent
Now we conclude the proof of the main result.
\begin{proof}[proof of Theorem \ref{thm: classification of PS sequences} ]
We define 
$$u^{1}_n:=u_n-u_{\infty}$$
and by Lemma \ref{lem: the energy splits} we have
$$E(u^{1}_n)=E(u_n)-E(u_\infty)+o(1) \;.$$
By the procedure described above, we find a sequence of points $\zeta_n^1$ converging to a concentration point $\zeta^1 \in S^{2N+1}$, a sequence of radii $R_n^1$ converging to zero, a solution $U_{\infty}^1$ to equation (\ref{eq: problem on H}) and a sequence $v_n^1$ defined as in (\ref{eq: v_n}); therefore we set:
\begin{equation*}
u_n^{2}:=u^{1}_n-v_n^{1}=u_n-u_{\infty}-v_n^{1}\;.
\end{equation*}
By Lemma \ref{lem: the energy splits2}, we get
\begin{equation*}
E(u_n^{2})=E(u^{1}_n)-E_{\H}(U_{\infty}^{1})+o(1)=E(u_n)-E(u_\infty)-E_{\H}(U_{\infty}^{1})+o(1).
\end{equation*}
Now we iteratively apply this procedure obtaining
\begin{equation*}
u_n^{m}=u_n-u_{\infty}-\sum_{l=1}^{m-1} v^{l}_n
\end{equation*}
and
\begin{equation*}
E(u_n^{m})=E(u_n)-E(u_{\infty})-\sum_{l=1}^{m-1} E_{\H}(U_{\infty}^{l})+o(1)\;.
\end{equation*}
Since $E_{\H}(U_{\infty}^l)\geq C_E$ for every $l=1\dots m$, we stop the process when $c-mC_E<C_E$. Indeed, by Lemma \ref{lem: critical energy level}, (PS) sequences at levels strictly below $\frac{k}{Q}C_{S}^{-\frac{Q}{2k}}$ converge strongly in $\SobSph$, and this concludes the proof.
\end{proof}


\section{Existence of infinitely many solutions}

\noindent
In this section we will prove the existence of infinitely many solutions of \eqref{eq: problem on H} proceeding as in \cite{MMT2015}. The key idea is to find a suitable subspace of the space of variations for the functional we are interested in, on which it is straightforward to perform the following minimax argument by Ambrosetti and Rabinowitz (see \cite[Theorems 3.13 and 3.14]{AR}).
\begin{lemma}\label{lem: Amb-Rab}
Let $X$ be a closed infinite dimensional subspace of $\SobSph$. Assume that $E_{|_{X}}$, the restriction of $E$ on $X$, satisfies the Palais-Smale compactness condition on $X$. Then, there exists a sequence $u_n$ of critical points for $E_{|_X}$ such that
\begin{equation*}
\int_{\Sph} |u_n|^{p^*}\dSph\to \infty \quad\textnormal{as } n\to \infty.
\end{equation*}
\end{lemma}

\noindent
Let  us start by fixing some notations. We denote by $\mathbb{O}(2N+2)$, the group of $(2N+2)\times(2N+2)$ orthogonal matrices, and
\begin{equation*}
\U=\{g\in \mathbb{O}(2N+2):\; gJ=Jg\},
\end{equation*}
where
\begin{equation*}
J=\begin{pmatrix}
0 & -I_{N+1} \\
I_{N+1} & 0
\end{pmatrix}.
\end{equation*}
Since the elements of $\U$ define isometries on $\Sph$ and $\LSph$ is an intertwining operator (i.e. satisfies \eqref{eq: intertwining}), it follows that the functional $E$ is invariant under the action of $\U$:
\begin{equation}\label{eq: E is U(N+1) invariant}
E(u)=E(u\circ g),\quad \textnormal{for all }g\in \U\;,
\end{equation}
(for a detailed proof see for instance \cite{Kr}). For a subgroup $G$ of $\U$ we define
\begin{equation*}
X_G=\left\{u\in \SobSph:\; u\circ g=u, \;\forall g\in G	\right\}.
\end{equation*}
\begin{lemma}\label{lem: E_XG is PS}
Let $G$ be a subgroup of $\U$ such that for any $\zeta_0\in \Sph$ the $G$-orbit of $\zeta_0$ has at least one accumulation point. Then, $E_{|_{X_G}}$, the restriction of $E$ to $X_G$, satisfies the Palais-Smale condition.
\end{lemma}
\begin{proof}
Let $u_n$ be a (PS) sequence for $E_{|_{X_G}}$ at level $c$.  By contradiction, we suppose that $u_n$ does not admit a converging subsequence in $X_G.$ Hence, by the classification of (PS) sequence given in Theorem \ref{thm: classification of PS sequences}, we deduce that the set of concentration points
\begin{equation*}
\Theta=\{\zeta^l\in\Sph:\; 1\leq l\leq m\}
\end{equation*}
is discrete, finite and non-empty. Here we have adopted the same notation used in Theorem \ref{thm: classification of PS sequences}. Let $\zeta_0\in\Theta$. Then, since the (PS) sequence $u_n$ is invariant under the action of $G,$ if $g_i$ with $i=1,\dots, j$, are $j$ fixed elements in $G$, we have that also $\zeta^i=g_i \zeta_0$ are concentration points belonging to $\Theta$. Again, by Theorem \ref{thm: classification of PS sequences} and recalling Remark \ref{rmk: Energy of the bubbles} we have
\begin{align}\label{eq: quantization of energy lem E_XG is PS}
c=\lim_{n\to\infty}E(u_n)=E(u_{\infty})+\sum_{i=1}^j E_{\H^N}(U^i_{\infty})\geq E(u_{\infty})+j\frac{k}{Q} C_{S}^{-\frac{Q}{2k}}.
\end{align}
On the other hand, by assumption, the $G$-orbit of $\zeta_0$ has an accumulation point, therefore $\Theta$ contains infinitely many concentration points of the type $\zeta_i=g_i\zeta_0.$ Hence, letting $j\to +\infty$ in \eqref{eq: quantization of energy lem E_XG is PS} we reach a contradiction.
\end{proof}

\noindent
We recall that quite a few examples of infinite dimensional subgroups of $\U$ satisfying hypotheses of Lemma \ref{lem: E_XG is PS} are provided in \cite{MMT2015} and \cite{Kr}. Now we prove our result:
\begin{proof}[Proof of Theorem \ref{thm: infinitelysolutions} ]
Let $G$ be a subgroup of $\U$ such that $X_G$ is an infinite dimensional vector space and suppose that for each $\zeta_0\in\Sph$ the $G$-orbit of $\zeta_0$ contains at least one accumulation point. By Lemma \ref{lem: E_XG is PS}, $E_{|_{X_G}}$ satisfies the Palais-Smale condition, therefore Lemma \ref{lem: Amb-Rab} allows to perform a minimax argument ensuring the existence of a sequence of critical points $u_n$ in $X_G$ for $E_{|_{X_G}}$ such that
\begin{equation}\label{eq: Lp norm to infty}
\int_{\Sph}|u_n|^{p^*}\dSph\to\infty\quad\textnormal{as }n\to\infty.
\end{equation}
Now, since the functional $E$ is invariant under the action of $G$, by the Principle of Symmetric Criticality (see \cite{Pa1979}), we have that any critical point of $E_{|_{X_G}}$ is also a critical point for $E$. Moreover, to each $u_n$, solution to \eqref{eq: problem on Sph}, corresponds a solution $U_n=\Jacc^{\frac{1}{p^*}}u_n\circ \Cay$ to \eqref{eq: problem on H} and \eqref{eq: Lp norm to infty} implies
\begin{align*}
\int_{\HN} U_n\LH U_n \dH &= \int_{\HN} \Jacc^{\frac{1}{p^*}}u_n\circ\Cay \LH\left(\Jacc^{\frac{1}{p^*}}u_n\circ\Cay \right)\dH\\
&=\int_{\Sph}u_n\LSph u_n\dSph\\
&=\int_{\Sph}|u_n|^{p^*}\dSph\to \infty\quad \textnormal{as } n\to\infty.
\end{align*}
But all the solutions to \eqref{eq: problem on H} of the type $\omega_{\lambda,\xi}$ have the same energy $\frac{k}{Q} C_{S}^{-\frac{Q}{2k}}$ (see Remark \ref{rmk: Energy of the bubbles}), consequently
\begin{equation*}
\int_{\HN} \omega_{\lambda,\xi}\LH \omega_{\lambda,\xi}\dH=\left(\frac{1}{2}-\frac{1}{p^*}\right)^{-1}E_{\H}(\omega_{\lambda,\xi})=\left(\frac{1}{2}-\frac{1}{p^*}\right)^{-1}\frac{k}{Q} C_{S}^{-\frac{Q}{2k}}.
\end{equation*}
So, it is clear that in the sequence $U_n$ (and hence in the sequence $u_n$) there are infinitely many solutions of \eqref{eq: problem on H} (or equivalently of (\ref{eq: problem on Sph})), distinct from $\omega_{\lambda,\xi}$.
\end{proof}


\appendix
\section{Appendix }
We recall here some definitions and properties for the Riesz potentials on Carnot groups. So let $\mathbb{G}$ be a Carnot group of homogeneous dimension $Q$ and $\Delta_{b}$ its sub-Laplacian, then we have
\begin{theorem}[\cite{FL}]
Let $0<\alpha<Q$ and consider $h(t,x)$ the fundamental solution of the operator $-\Delta_{b}+\frac{\partial}{\partial t}$, then the integral
$$R_{\alpha}(x)=\frac{1}{\Gamma(\frac{\alpha}{2})}\int_{0}^{\infty}t^{\frac{\alpha}{2}-1}h(t,x)dt$$
converges absolutely and it satisfies the following properties:
\begin{itemize}
\item $R_{\alpha}$ is a kernel of type $\alpha$. In particular it is homogeneous of degree $\alpha-Q$
\item $R_{2}$ is the fundamental solution of $-\Delta_{b}$
\item $R_{\alpha}*R_{\beta}=R_{\alpha+\beta}$ for $\alpha$ and $\beta>0$ and $\alpha+\beta<Q$.
\item For $f\in L^{p}(\mathbb{G})$ and $1<p<\infty$, we have that 
$$(-\Delta_{b})^{-\frac{\alpha}{2}}f=f*R_{\alpha}.$$
\end{itemize}
\end{theorem}

\noindent
In this paper, we used the convention
$$R_{\alpha}f:=(-\Delta_{b})^{-\frac{\alpha}{2}}f=f*R_{\alpha}.$$
From the integral form of $R_{\alpha}$ one has
$$R_{\alpha}(x)\approx |x|^{-Q+\alpha}$$
and $\rho(x):=(R_{\alpha}(x))^{\frac{1}{\alpha-Q}}$ defines a $\mathbb{G}$-homogeneous norm, smooth away from the origin and it induces a quasi-distance that is equivalent to the left-invariant Carnot-Caratheodory distance. In a similar way, one can define the function $\tilde{R}_{\alpha}$, introduced in \cite{Fran}, for $\alpha<0$ and $\alpha \not \in \{0,-2,-4,\dots\}$ by
$$\tilde{R}_{\alpha}(x)=\frac{\frac{\alpha}{2}}{\Gamma(\frac{\alpha}{2})}\int_{0}^{\infty}t^{\frac{\alpha}{2}-1}h(t,x)dt$$
Again, it is easy to see that $\tilde{R}_{\alpha}$ is $\mathbb{G}$-homogeneous of degree $\alpha-Q$ and
$$\tilde{R}_{\alpha}(x)\approx |x|^{\alpha-Q}.$$
Using this function, it is possible to define another representation for the fractional sub-Laplacian, which we use in the proofs of our results.
\begin{theorem}[\cite{Fran}]
If $u$ is a Schwartz function on $\mathbb{G}$, then for $0<\alpha<2$ one has
$$(-\Delta_{b})^{\frac{\alpha}{2}}u(x)=PV\int_{\mathbb{G}}(u(y)-u(x))\tilde{R}_{-\alpha}(y^{-1}x)dy$$
\end{theorem}

\noindent
Using classical interpolation (or what is it called $\lambda$-kernel estimates in \cite{FS}) one has for $0<\alpha<Q$,
\begin{equation}
\|R_{\alpha}u\|_{p}\lesssim \|u\|_{q},
\end{equation}
for $\frac{1}{p}=\frac{1}{q}-\frac{\alpha}{Q}$ and $1<q<\infty$. In the case of the Heisenberg group one in fact has more explicit computations for the operator $\LH$ (see \cite{Ls}). In fact, one can replace $R_{\alpha}$ by the expected Green's function of $\LH$ that is $$G_{2k}(x)=c_{n,k}\frac{1}{|x|^{Q-2k}}$$
and $\tilde{R}_{-\alpha}$ by the kernel $$K_{2k}=\tilde{c}_{n,k}\frac{1}{|x|^{Q+2k}}.$$
The integral representation formula holds also for the operator $\LH$ in our results, with $\tilde{R}_{-2k}$ replaced by $K_{2k}$.


\end{document}